\documentclass[11pt]{article}
\usepackage{lineno,hyperref,amsfonts,amsmath}
\usepackage{appendix}
\usepackage[numbers,sort&compress]{natbib}
\modulolinenumbers[5]
\numberwithin{equation}{section}
\usepackage{indentfirst}
\usepackage[a4paper,left=2.8cm,right=2.8cm,top=2.54cm,bottom=2.54cm]{geometry}%
\usepackage{amssymb}
\usepackage{latexsym}
\usepackage{graphicx} 
\usepackage{epstopdf}
\graphicspath{{figures/}}
\usepackage{mathrsfs}
\usepackage{bm}
\usepackage{cases}
\usepackage{dsfont}
\usepackage{tikz}
\usepackage{threeparttable}
\usepackage{subfigure}
\usepackage{float}
\usepackage{makecell,multirow,diagbox}
\usetikzlibrary{shapes,snakes}

\newcommand{\tnormtight}[2][]{\vert\!\vert\!\vert #2 \vert\!\vert\!\vert_{#1}}

\usepackage{amsthm}
\theoremstyle{plain}
\newtheorem{assumption}{Assumption}[section]
\newtheorem{theorem}{Theorem}[section]
\newtheorem{lemma}{Lemma}[section]

\numberwithin{equation}{section}

\numberwithin{figure}{section}

\numberwithin{table}{section}
\theoremstyle{definition}

\theoremstyle{remark} 
\newtheorem{re}{Remark}
\usepackage[]{caption2}

\usepackage{multirow} 
\usepackage{yhmath}
\usepackage{subfigure}
\usepackage{mathrsfs}
\usepackage{tabularx}
\usepackage[boxed,ruled,commentsnumbered]{algorithm2e}

\newcommand{\normmm}[1]{{\left\vert\kern-0.25ex\left\vert\kern-0.25ex\left\vert #1
		\right\vert\kern-0.25ex\right\vert\kern-0.25ex\right\vert}}
\usepackage{accents}

\begin{document}
	\title {A parallel solver for random input problems via Karhunen-Lo\`{e}ve expansion and diagonalized coarse grid correction}
    \author{
Dou Dai\thanks{School of Mathematics, Hunan University, Changsha 410082, China. Email: ddou1008@hnu.edu.cn}
\and
Qiuqi Li\thanks{School of Mathematics, Hunan University, Changsha 410082, China. Email: qiuqili@hnu.edu.cn}
\and
Huailing Song\thanks{School of Mathematics, Hunan University, Changsha 410082, China. Email: shling@hnu.edu.cn}
}
\date{}
	
	\maketitle
	\begin{center}
		\textbf{Abstract}
        \end{center}
        This paper is dedicated to enhancing the computational efficiency of traditional parallel-in-time methods for solving stochastic initial-value problems. The standard parareal algorithm often suffers from slow convergence when applied to problems with stochastic inputs, primarily due to the poor quality of the initial guess. To address this issue, we propose a hybrid parallel algorithm, termed KLE-CGC, which integrates the Karhunen-Lo\`{e}ve (KL) expansion with the coarse grid correction (CGC). The method first employs the KL expansion to achieve a low-dimensional parameterization of high-dimensional stochastic parameter fields. Subsequently, a generalized Polynomial Chaos (gPC) spectral surrogate model is constructed to enable rapid prediction of the solution field. Utilizing this prediction as the initial value significantly improves the initial accuracy for the parareal iterations. A rigorous convergence analysis is provided, establishing that the proposed framework retains the same theoretical convergence rate as the standard parareal algorithm. Numerical experiments demonstrate that KLE-CGC maintains the same convergence order as the original algorithm while substantially reducing the number of iterations and improving parallel scalability.

	\section{Introduction}

Parametric partial differential equations (PDEs) represent fundamental mathematical models for complex physical systems characterized by uncertain or adjustable parameters,  such as material properties, boundary conditions, and external loads. These equations are crucial in fields including constrained optimization, feedback control formulations, and stochastic problems in uncertainty quantification. However, the efficient numerical solution of these equations is hindered by the so-called double curse of dimensionality. Firstly, high-dimensional parameter spaces cause the computational cost of traditional sampling methods like Monte Carlo to grow exponentially, constituting the parametric curse of dimensionality. Secondly, for time-dependent problems, simulating long-time evolution is constrained by the sequential nature of time integration. In such cases, the sequential requirement of time integration imposes a significant constraint. Although spatial parallelism can be exploited, the inherent serial nature of time advancement presents a fundamental bottleneck that ultimately limits achievable speedup.

To overcome the limitation in the temporal dimension, Lions et al. proposed the parareal algorithm \cite{J.L.L}. Its core idea is to decompose the global time interval $[0, T]$ into $N$ coarse subintervals with a large step-size $\Delta T$, enabling temporal parallelism. The algorithm then employs an iterative correction between two numerical propagators: a high-fidelity fine propagator $\mathcal{F}$ that advances the solution with a small step-size $\Delta t$ ($\Delta T / \Delta t = J \geq 2$) on each subinterval, and a low-cost coarse propagator $\mathcal{G}$ that operates on the coarse time grid with step-size $\Delta T$. However, a key limitation of the traditional parareal algorithm lies in its coarse grid correction (CGC) step, which relies on the sequential update of initial values on the coarse time grid, making CGC a critical bottleneck for parallel speedup \cite{M.J,S.L1}.

To gain a deeper understanding and optimize the algorithm, Gander and Vandewalle established a theoretical framework for the convergence analysis of the parareal algorithm in~\cite{M.J}. For a scalar linear model problem, they revealed a direct relationship between the algorithm's error contraction factor and the stability functions $\mathcal{R}_g$, $\mathcal{R}_f$ of the propagators $\mathcal{G}$ and $\mathcal{F}$:
\begin{equation}
\mathcal{K}(z,J) = \frac{|\mathcal{R}_g(z) - \mathcal{R}_f^J(z/J)|}{1 - |\mathcal{R}_g(z)|}, \quad \text{where } z = \Delta T \lambda.
\end{equation}
This result provides an important theoretical tool for subsequent research. Mathew et al.~\cite{T.R.M} analyzed the convergence when both $\mathcal{G}$ and $\mathcal{F}$ employ the backward Euler method for symmetric positive definite systems, proving a convergence factor $\rho \approx 1/3$. Wu~\cite{S.L2,S.L3} extended the convergence analysis to higher-order time integrators such as second- and third-order SDIRK methods, demonstrating that under certain conditions on the grid ratio $J$, the convergence factor can remain robust, independent of spatial and temporal discretization parameters.

To fundamentally overcome the serial bottleneck of CGC, Wu~\cite{S.L} proposed a parallel CGC method based on matrix diagonalization. The innovation lies in applying the coarse propagator $\mathcal{G}$ to a modified model:
\begin{equation}
u'(t) + Au(t) = f, \quad u(0) = \alpha u(T).
\end{equation}
By introducing the parameter $\alpha$ to couple the initial and final values, the originally block lower triangular CGC system matrix can be diagonalized, transforming the correction step into a fully parallel form. Theoretical analysis shows that when $|\alpha| \leq \alpha^*$, the convergence rate of this parallel CGC algorithm is consistent with the classical Parareal algorithm, and the condition number of the diagonalized matrix is $\mathcal{O}(1)$, independent of the time interval length $T$, thus effectively controlling round-off errors~\cite[Theorem 3.2, Lemma 2.2]{S.L}. This work provides key technical support for the practical and efficient application of parallel-in-time (PinT) algorithms.

The application scope of the parareal algorithm continues to expand. Addressing the challenging problem of time-dependent diffusion equations with fractional Laplacian operators, Wu~\cite{Ref2} conducted an in-depth analysis of an efficient parareal algorithm. This algorithm employs a third-order SDIRK method as the fine propagator $\mathcal{F}$ to ensure high accuracy, and innovatively adopts an implicit-explicit (IMEX) Euler method as the coarse propagator $\mathcal{G}$. Furthermore, the convergence analysis was successfully extended to time-periodic problems, demonstrating the flexibility of the Parareal algorithm in handling different boundary conditions. In addition to the deterministic approaches based on numerical analysis discussed above, recent advances in probabilistic numerical methods offer a new paradigm for accelerating Parareal convergence. Pentland et al.~\cite{Penk} proposed a stochastic parareal algorithm. This method departs from the traditional deterministic approach of propagating a single initial value. Instead, on each unconverged time subinterval, it constructs a probability distribution based on current iterative information and draws several candidate initial values from it. All these candidate values are then propagated in parallel using the expensive $\mathcal{F}$ propagator. Subsequently, the algorithm selects the sample that yields the most continuous trajectory in phase space between adjacent subintervals, thereby providing a superior initial guess for the prediction-correction step.

Beyond the diagonalization-based and probabilistic methods mentioned, other significant PinT methods have been developed, such as parallel full approximation scheme in space-time (PFASST) ~\cite{M.L.M}, multigrid reduction in time (MGRiT)~\cite{R.D.F,V.A.D}, and the adaptive parareal algorithm~\cite{Y.M}. These methods further extend the parallel capability and efficiency of PinT across various application scenarios. Although parallel CGC technology--whether based on diagonalization or stochastic sampling--significantly enhances temporal parallel capability, its iterative convergence speed remains strongly dependent on the quality of the initial guess. A poor initial guess increases the number of iterations, thereby offsetting some of the parallel gains.

For dimensionality reduction in the parameter space, the Karhunen-Loève (KL) expansion and the generalized Polynomial Chaos (gPC) method are two prominent techniques. The KL expansion achieves a low-dimensional representation by extracting dominant orthogonal modes from the covariance structure of parameterized random fields \cite{C.S}. The gPC method employs orthogonal polynomial bases to expand stochastic functions into a series of deterministic coefficients, thereby transforming stochastic problems into deterministic systems \cite{D.X1, J.S}. Early research primarily applied KL and gPC expansions to approximate parameterized coefficient fields \cite{R.G, D.X2}, providing an effective means of input reduction for uncertainty propagation. Although this approach offers significant value, its applicability in many-query scenarios--such as design optimization or parameter inversion--suffers from inherent limitations. Specifically, solving the full PDE system for each new parameter restricts computational efficiency, as it fails to fully leverage the reuse potential of precomputed information. Furthermore, controlling the propagation of approximation errors from the coefficient field through the subsequent PDE solution process necessitates careful management, increasing the complexity of overall error control.

A more direct and efficient strategy for many-query scenarios is to construct approximations directly from the solution field of the PDE. This approach aims to build a compact surrogate model of the solution manifold. The core insight of this work is to leverage this strategy not for full field reconstruction, but for a critical yet underexplored purpose: generating high-quality initial guesses for iterative PinT algorithms.

The main contribution of this work is a novel framework that synergistically integrates the diagonalization-based parallel CGC algorithm with a KL-gPC-based method for initial value approximation. Instead of starting the iterative CGC process with a generic initial guess, our framework utilizes a precomputed solution field library to construct a high-fidelity initial value for each new parameter query via the KL-gPC method. The workflow consists of two stages: first, the diagonalization-based parallel CGC algorithm is employed to efficiently generate a high-fidelity solution field library for a set of sampled parameters, leveraging temporal parallelism. Then, KL and gPC expansions are applied to this library to build a surrogate model that maps any new parameter to a high-quality initial guess for the CGC iteration. Since the quality of the initial value directly impacts the convergence rate, this tailored initial approximation can effectively reduce the number of iterations required, thereby achieving a synergistic optimization where temporal parallelism and improved iterative convergence compound to enhance overall efficiency.

The paper is structured as follows: Section \ref{Pre} introduces the mathematical formulation of parametric PDEs and the fundamentals of KL expansion and gPC reduction. Section \ref{KLE-C} provides a detailed exposition of the Parareal method, the diagonalization-based parallel CGC method, and the implementation of the proposed KL-gPC-enhanced initial value generation for the CGC algorithm. Section \ref{cov} presents an analysis of the proposed method, focusing on how the improved initial value influences convergence. Section \ref{nu} validates the efficiency and accuracy of the method through numerical examples, demonstrating the reduction in iteration count compared to the standard parallel CGC approach. Finally, Section \ref{con} concludes the paper and discusses future research directions.

	\section{Preliminaries and notations}\label{Pre}
    In this section, we present some preliminaries and notations for the rest of paper. 
    
    Let $\boldsymbol{\xi} = (\xi_1,\dots,\xi_{n_t})^{T}$ be a vector of random variables parameterizing the uncertainties of the input fields defined on a space $(\mathcal{S}, \mathcal{F}, \mathcal{P})$, where $\mathcal{S}$ is the set of events, $\mathcal{F}$ is the $\sigma$-algebra and $\mathcal{P}$ is the probability measure. We will assume that the entries of $\boldsymbol{\xi}$ are independent and identically distributed. 
     Let $\Omega\subset \mathbb{R}^d$ be a bounded domain with Lipschitz continuous boundary $\partial\Omega$. 
    The associated norm is given by $\|\cdot\|_{X(\Omega)}=\sqrt{(\cdot,\cdot)_{X(\Omega)}}$.
	
	Let $v(x,t;\boldsymbol{\xi}): \Omega\times[0,T]\times\mathcal{S}\rightarrow \mathbb{R}$ represents a real-valued random field and $P$ denotes the joint probability density function of
	$\boldsymbol{\xi}$. We define the Hilbert space  $L^2_P(\Xi)$ of the random variables with second-order moments as follows
	\begin{equation*}
		L^2_P(\Xi):=\{v: \boldsymbol{\xi}\in\Xi\mapsto v(\boldsymbol{\xi})\in \mathbb{R}; \quad \int_{\Xi}u(\boldsymbol{\xi})^2P(d\boldsymbol{\xi})<\infty\}.
	\end{equation*}
	The inner product of this space $L^2_P(\Xi)$ is given by
	\begin{equation*}
		(u,v)_{L^2_P(\Xi)}:=\int_{\Xi}u(\boldsymbol{\xi})v(\boldsymbol{\xi})P(d\boldsymbol{\xi}),
	\end{equation*}
	which induces the norm $\|v\|_{L^2}=\|v\|_{L^2_P(\Xi)}:=\sqrt{(v,v)_{L^2_P(\Xi)}}$.
	
Consider the following parameter-dependent dynamical systems: 
	\begin{equation}
		\left\{ 
		\begin{aligned}
			u'(x,t;\boldsymbol{\xi})&= f(u(x,t;\boldsymbol{\xi}),t;\boldsymbol{\xi}), \quad t\in[0,T],
			\\
			u(x,0;\boldsymbol{\xi})&= u_0(\boldsymbol{\xi}),
		\end{aligned}
		\right.
		\label{eq:psys}
	\end{equation}
	 where the flux $f$ and initial condition depend on some parameters $\boldsymbol{\xi}$ . The solution $ u(x,t;\boldsymbol{\xi})$ belongs to the state space $X(\Omega)$.
We remark that the fine-grid solution is considered as a reference solution in the paper.

Solving the PDE (\ref{eq:psys}) anew for each parameter leads to a prohibitive computational cost that scales exponentially with both the parameter dimension and the spatial-temporal grid density, thus failing to meet the efficiency requirements of many-query tasks like design optimization and uncertainty quantification. The pursuit of computational efficiency therefore centers on constructing a separable representation of the input-output relationship through variable separation—an approach justified by the intrinsic structure of the parameterized solution field
$u(t,\boldsymbol{\xi})$. Specifically, the field's evolution is often characterized by a few dominant modes, which exhibit commonality across parameters. The KL expansion exploits this structure to achieve variable separation, with the goal of extracting the orthogonal principal components of $u(t,\boldsymbol{\xi})$ in the spatial-temporal domain and constructing an approximation of the form
\begin{eqnarray}
\label{separating presentation}
u(x, t;\boldsymbol{\xi})\approx u_{M_Q}(x,t;\boldsymbol{\xi}):=\sum_{i=1}^{M_Q}\zeta_i(\boldsymbol{\xi})g_i(x,t),
\end{eqnarray}
where  $\zeta_i(\boldsymbol{\xi})$ only depends on $\boldsymbol{\xi}$ and $g_i(x,t)$ only depends on $x$ and $t$.

We assume that $X_{M_Q}(\Omega) \subset X(\Omega)$ is a finite dimensional subspace space, and $\{\Psi_i\}_{i=1}^{M_Q}$ is a set of basis functions for $X_{M_Q}(\Omega) $.
We want to find an approximation of $u(x, t;\boldsymbol{\xi})$ in $X_{M_Q}(\Omega)$ such that
\begin{eqnarray}
\label{output-approximation}
 \|u(x,t;\boldsymbol{\xi})-\sum_{j=1}^{M_Q} c_{j}\Psi_j(x,t;\boldsymbol{\xi}) \|_{X(\Omega)}\leq \delta,
\end{eqnarray}
where $\delta$ is a given threshold.

\subsection{Karhunen-Lo\`{e}ve expansion}\label{sec:kl_exp}

In this subsection, we introduce a KL expansion method of snapshots to get the approximation (\ref{output-approximation}).
Let $\Xi_{t}$ be a collection of a finite number of samples in $\Xi$ and the cardinality $|\Xi_{t}|=n_t$.
For $\forall ~\boldsymbol{\xi}\in \Xi_{t}$, we can split $u(x,t;\boldsymbol{\xi})$ into two parts, i.e.,
\[
u(x,t;\boldsymbol{\xi})=\bar{u}(x,t)+\tilde{u}(x,t;\boldsymbol{\xi}),
\]
where $\bar{u}(x,t):=E[u(x,t;\boldsymbol{\xi})]=\frac{1}{n_{t}}\sum_{i=1}^{n_{t}}u(x,t;\boldsymbol{\xi}_i)$ is the mean, and $\tilde{u}(x,t,\boldsymbol{\xi})=u(x,t;\boldsymbol{\xi})-\bar{u}(x,t)$ is a random fluctuating part.
To obtain $\tilde{u}(x,t;boldsymbol{\xi})$, we take a set of snapshots $\{\tilde{u}(x,t;\boldsymbol{\xi}_i)\}_{i=1}^{n_{t}}$ and compute a covariance matrixes $\textbf{C}$, whose entries can be defined by
\[
\textbf{C}_{n,m}:=\frac{1}{n_{t}}\bigg(\tilde{u}(x,t;\boldsymbol{\xi}_n),\tilde{u}(x,t;\boldsymbol{\xi}_m)\bigg)_{X(\Omega)}.
\]
Let $\{\hat{\lambda}_k,  \textbf{e}_k\}$  be the eigen-pairs (normalized)  of $\textbf{C}$, $1\leq k\leq n_{t}$. Set $(\textbf{e}_k)_j=e_k^{j}$, we define the functions
\begin{eqnarray}
\label{KLE basis}
g_k(x,t):=\frac{1}{\sqrt{\hat{\lambda}_k n_{t}}}\sum_{j=1}^{n_{t}}e_k^{j}\tilde{u}(x,t;\boldsymbol{\xi}_j).
\end{eqnarray}
It is easy to get $(g_k,g_l)_\mathcal{H}=\delta_{k,l}$, $1\leq k$, $l\leq n_{t}$. Then we have the following
\[
\tilde{u}(x,t;\boldsymbol{\xi})\approx \sum_{i=1}^{ n_{t}} \sqrt{\hat{\lambda}_i} \zeta_i(\boldsymbol{\xi}) g_i(x,t),
\]
where $\{\zeta_i(\boldsymbol{\xi})\}_{i=1}^{n_{t}}$  are given by
\begin{eqnarray}
\label{KLE parameter}
\begin{split}
\zeta_i(\boldsymbol{\xi}):= \frac{1}{\sqrt{\hat{\lambda}_i}} \big(\tilde{u}(\cdot,\cdot;\boldsymbol{\xi}),g_i\big)_{X(\Omega)}.
\end{split}
\end{eqnarray}
Thus we get the decomposition
\begin{eqnarray}\label{KLE Snapshots}
u(x,t;\boldsymbol{\xi})\approx \bar{u}(x,t)+\sum_{i=1}^{M_Q} \sqrt{\hat{\lambda}_i}\zeta_i(\boldsymbol{\xi}) g_i(x,t).
\end{eqnarray}

Since equation (\ref{KLE parameter}) involves $\tilde{G}(\cdot,\cdot,\boldsymbol{\xi})$, it cannot be directly employed to compute the functions $\{\zeta_i(\boldsymbol{\xi})\}_{i=1}^{n_{t}}$ for arbitrary parameter values $\boldsymbol{\xi}$. To overcome this, we adopt a least-squares approach using orthogonal polynomials to approximate $\{\zeta_i(\boldsymbol{\xi})\}_{i=1}^{n_{t}}$.

Let $\{\psi_i(\boldsymbol{\xi})\}_{i=1}^{P}$ denote the set of orthogonal polynomial basis functions defined over the parameter space of $\boldsymbol{\xi}$. These basis functions are ordered sequentially and arranged into a row vector as follows:
 \[
 \big[\psi_1(\boldsymbol{\xi}), \psi_2(\boldsymbol{\xi}),\cdots,\psi_{P}(\boldsymbol{\xi})\big].
 \]
 For the sample date  $\Xi_{t}$, we compute $[\psi_1(\boldsymbol{\xi}_j), \psi_2(\boldsymbol{\xi}_j),\cdots,\psi_{P}(\boldsymbol{\xi}_j)]$ and $\zeta_i(\boldsymbol{\xi}_j)=\frac{1}{\sqrt{\hat{\lambda}_i}} (\tilde{Q}(\cdot,\boldsymbol{\xi}_j),g_i)_{X(\Omega)}$ ($j = 1,\cdots,n_{t}$).  They are putted in the following matrix $\textbf{B}$ and vector $\textbf{D}$ , respectively,
 \begin{eqnarray}
\label{least square A}
 \textbf{B}:=\left[ \begin{array}{ccc}  \psi_1(\boldsymbol{\xi}_1)& \ldots & \psi_{P}(\boldsymbol{\xi}_1)\\ \vdots &\ddots &\vdots\\ \psi_1(\boldsymbol{\xi}_{n_{t}})& \ldots & \psi_{P}(\boldsymbol{\xi}_{n_{t}})\\\end{array}\right] ,
 \end{eqnarray}
 \begin{eqnarray}
\label{least square F}
 \textbf{D}:=[\zeta_i(\boldsymbol{\xi}_1) \cdots  \zeta_i(\boldsymbol{\xi}_{n_{t}})]^{T}.
 \end{eqnarray}
 We obtain the approximation of the parameter functions $\zeta_i(\boldsymbol{\xi})$ by solving the following least square problem,
\begin{eqnarray}
\label{least square problem}
\begin{split}
\textbf{h}=\arg\min_{\beta}\|\textbf{B}\beta-\textbf{D}\|_2.
\end{split}
\end{eqnarray}
 Thus we get $\zeta_i(\boldsymbol{\xi})\approx \sum\limits_{i=1}^{P}h_i\psi_i(\boldsymbol{\xi})$, and $h_i=(\textbf{h})_i$.
  

\subsection{Generalized polynomial chaos formulation}
In the previous section, while presenting the complete framework of the KL expansion, the discussion of orthogonal polynomials did not extend to the specific selection criteria, accuracy control, or compatibility with PinT frameworks. In essence, the KL expansion accomplishes dimensionality reduction by transforming a high dimensional random field into a set of low dimensional random coefficients. However, to further convert these coefficients into deterministic information suitable for integration into temporal parallel computations, a systematic spectral representation method is required: gPC serves as the key tool for achieving this transformation.
 For practical implementation, we begin by constructing appropriate basis functions and establishing the expansion framework.

For practical implementation, we begin by constructing appropriate basis functions and establishing the expansion framework.
For each random variable $\xi_l$ where $l=1,\dots,N_p$, we define a univariate orthogonal polynomial basis $\{ \psi_{k}(\xi_{l}) \}_{k=0}^{\infty}$ that satisfies the orthonormality condition:
\[
\langle \psi_{i}(\xi_{l}), \psi_{j}(\xi_{l}) \rangle = \delta_{ij}, \quad \text{for } l=1,\dots, N_p.
\]

The multivariate basis functions are constructed via tensor products:
\[
\psi_{\footnotesize \bm{k}}(\boldsymbol{\xi}) = \prod_{l=1}^{N_p}\psi_{k_l}(\xi_l),
\]
where $\bm{k} = (k_1,\cdots,k_{N_p}) \in \mathbb{N}_{0}^{N_p}$ is a multi-index. To enable computational implementation, we truncate this infinite basis using the total-order index set:
\begin{equation}
\label{eq:complete}
\Lambda_{p,n_t} = \left\{\bm{k}\in\mathbb{N}_{0}^{N_p}: \Vert \bm{k}\Vert_{1}\leq p\right\},
\end{equation}
which contains $P = \frac{(p+N_p)!}{p!N_p!}$ basis functions. For notational convenience, we employ single-index notation $\{\psi_k\}_{k=1}^P$ in subsequent discussions.

Any function $\zeta(\boldsymbol{\xi}) \in L^2(\Omega)$ admits the gPC expansion:
\begin{equation}
\label{eq:def_gpc}
\zeta(\boldsymbol{\xi}) = \sum_{\bm{k} \in \mathbb{N}_{0}^{N_p} }h_{\bm{k}}\psi_{\bm{k}}(\boldsymbol{\xi}) \approx \sum_{k=1}^{P}h_{k}\psi_{k}(\boldsymbol{\xi}),
\end{equation}
with expansion coefficients determined by orthogonal projection:
\begin{equation}
\label{eq:proj_coe}
h_{\bm{k}} = \left\langle \zeta, \psi_{\bm{k}}(\boldsymbol{\xi})\right\rangle.
\end{equation}
This expansion exhibits spectral convergence:
\[
\lim_{P \rightarrow \infty}\left\| \zeta - \sum_{k=1}^{P}h_{k}\psi_{k}(\boldsymbol{\xi}) \right\|_{L^2(\Omega)} = 0.
\]

The evaluation of projection coefficients (\ref{eq:proj_coe}) requires multidimensional numerical integration. While tensor products of univariate quadrature rules (e.g., Gauss quadrature) are possible, they become computationally prohibitive in high dimensions. Sparse quadrature techniques, such as Smolyak's rule \citep{Smolyak_63}, offer a more efficient alternative by selectively combining points from lower-dimensional tensor products.

To avoid repeated stochastic integration at each spatiotemporal grid point, we precompute all stochastic integrals and derive an extended deterministic system through stochastic projection, resulting in an extended initial-boundary-value problem.

The choice of polynomial basis depends on the probability distribution of $\boldsymbol{\xi}$. Table \ref{tab:wiener_askey_correspondence} summarizes common correspondences, where continuous distributions employ Hermite, Laguerre, or Jacobi polynomials, while discrete distributions utilize Charlier, Krawtchouk, Meixner, or Hahn polynomials. Notably, Legendre polynomials (a special case of Jacobi polynomials $P_n^{(\alpha,\beta)}(x)$ with $\alpha=\beta=0$) correspond to uniform distributions and are listed separately due to their practical importance.

\begin{table}[htbp]
\centering
\caption{Correspondence between common probability distributions and their associated gPC basis functions}
\centering
	\medskip\small\renewcommand{\arraystretch}{1.15}
\label{tab:wiener_askey_correspondence}
\begin{tabular}{c|c|c|c}
\hline
& \textbf{Random inputs} & \textbf{gPC basis polynomials} & \textbf{Support} \\
\hline
\multirow{4}{*}{Continuous} 

& Gaussian       & Hermite   & $(-\infty,\infty)$ \\
& Gamma          & Generalized Laguerre  & $[0,\infty)$ \\
& Beta           & Jacobi    & $[a,b]$ \\
& Uniform        & Legendre  & $[a,b]$ \\
\hline
\multirow{4}{*}{Discrete} 
& Poisson        & Charlier  & $\{0,1,2,\ldots\}$ \\
& Binomial       & Krawtchouk & $\{0,1,\ldots,N\}$ \\
& Negative binomial & Meixner  & $\{0,1,2,\ldots\}$ \\
& Hypergeometric & Hahn      & $\{0,1,\ldots,N\}$ \\
\hline
\end{tabular}
\end{table}

        \section{Karhunen-Lo\`{e}ve expansion with parallel coarse-grid correction method }\label{KLE-C}
       In this section, we detail the classical parareal algorithm and the diagonalization technique for parameter-dependent dynamical systems. To further enhance computational efficiency, we integrate KL expansion with the diagonalization-based parallel CGC (KLE-PCGC) method to address problems involving random inputs.
		\subsection{The parareal algorithm}
		For parameter-dependent dynamical systems \eqref{eq:psys}, we decompose the time interval $[0,T]$ into $N$ equal subintervals $[T_n,T_{n+1}]$, $n=0,1,\cdots,N-1$, with $0=T_0<T_1<\cdots<T_{N-1}<T_N=T$, $\Delta T := T_n - T_{n-1}$, and consider the $N$ separate initial value problems
		\begin{equation}
			\left\{ 
			\begin{aligned} 
				&u'_n(t,\boldsymbol{\xi}) = f\big( u_n (t, \boldsymbol{\xi} ), t, \boldsymbol{\xi} \big), \quad t \in [T_n, T_{n+1}], \\
				&u_n (T_n,\boldsymbol{\xi}) = \mathrm{U}_n(\boldsymbol{\xi}).
			\end{aligned}
			\right.
			\label{eq:subprobs}
		\end{equation}
		Each solution $u_n(t, \boldsymbol{\xi})$ is defined over $[T_n,T_{n+1}]$ given the initial values $\mathrm{U}_n( \boldsymbol{\xi}) \in \mathbb{R}^d$ at $t = T_n$. Note however that only the initial value $\mathrm{U}_0(\boldsymbol{\xi}) = u_0(\boldsymbol{\xi})$ is known, whereas the rest ($\mathrm{U}_n(\boldsymbol{\xi})$ for $n \geq 1$) need to be determined before \eqref{eq:subprobs} can be solved in parallel. These initial values must satisfy the continuity conditions
		\begin{equation} \label{eq:exact}
			\mathrm{U}_0(\boldsymbol{\xi}) = u_0(\boldsymbol{\xi}) \quad \text{and} \quad \mathrm{U}_{n+1}(\boldsymbol{\xi}) = u_{n}(T_{n+1}, \boldsymbol{\xi}) \quad \text{for} \quad n = 0,1,\dots,N-1,
		\end{equation}
		which form a (nonlinear) system of $N+1$ equations that ensure solutions match at each $T_n ( \forall n \geq 1)$. System \eqref{eq:exact} is solved for $\bm{u}_n(\boldsymbol{\xi})$ using the Newton-Raphson method to form the iterative system
		\begin{subequations} \label{eq:iterative}
			\begin{align} 
				\mathrm{U}^{k+1}_0 (\boldsymbol{\xi})&= u_0(\boldsymbol{\xi}), \label{eq:iterative_a} \\ 
				\mathrm{U}^{k+1}_{n+1}(\boldsymbol{\xi}) &= u_{n} (T_{n+1}, \boldsymbol{\xi}) + \frac{\partial u_{n}}{\partial \mathrm{U}_{n}} (T_{n+1}, \boldsymbol{\xi}) \big[\mathrm{U}^{k+1}_{n}(\boldsymbol{\xi}) - \mathrm{U}^{k}_{n}(\boldsymbol{\xi}) \big] , \label{eq:iterative_b}
			\end{align}
		\end{subequations}
		for $n = 1,\dots,N$, where $k = 0,1,2,\dots$ is the iteration number. This system contains the unknown solutions $u_n(\boldsymbol{\xi})$ and their partial derivatives, which even if known, would be computationally expensive to calculate.
		
		To solve \eqref{eq:iterative}, the parareal algorithm utilizes two numerical integrators. The first is a numerically fast $\mathcal{G}$-propagator, $\mathcal{G}(T_n, T_{n+1}, \mathrm{U}_n^k(\boldsymbol{\xi}))$ denotes integrates over the interval $[T_n,T_{n+1}]$ using initial values $\mathrm{U}_n^k(\boldsymbol{\xi})$. The second is a $\mathcal{F}$-propagator, which runs significantly slower than $\mathcal{G}$ but offers much greater numerical accuracy;  $\mathcal{F}(t_{n,j}^k, t_{n,j+1}^k, \mathrm{U}_n^k(\boldsymbol{\xi}))$ is designed to integrate over the interval $[t_{n,j}^k, t_{n,j+1}^k]$ with initial values $\mathrm{U}_n^k(\boldsymbol{\xi})$, which the $\mathcal{F}$-propagator runs $J_n^k$ steps in total within the large subinerval $[T_n,T_{n+1}]$. In our implementation, the distinction between fast and slow integration is ensured by setting the time steps for $\mathcal{G}$ and $\mathcal{F}$ as $\Delta T$ and $\Delta t$, respectively, with $\Delta t \ll \Delta T$. 
		The key principle is that if $\mathcal{F}$ were used to integrate \eqref{eq:psys} over $[T_n,T_{n+1}]$ serially, it would require an infeasible amount of computational time, highlighting the necessity of using parareal algorithm. Therefore, $\mathcal{G}$ is allowed to run serially across multiple subintervals rapidly, while the slower solver $\mathcal{F}$ is restricted to running in parallel on subintervals. This is a strict requirement for the effective execution of the parareal algorithm; otherwise, numerical speedup cannot be achieved. The result is that an initial guess for the initial values $\mathrm{U}^{0}_n(\boldsymbol{\xi})$ (found using $\mathcal{G}$) is improved at successive parareal iterations $k$ using the CGC
		\begin{equation} \label{eq:pred_correc}
			\mathrm{U}^{k+1}_{n+1}(\boldsymbol{\xi}) = \mathcal{G}\big(T_n, T_{n+1}, \mathrm{U}^{k+1}_{n}(\boldsymbol{\xi})\big) + \mathcal{F}\big(T_n, T_{n+1}, \mathrm{U}^{k}_{n}(\boldsymbol{\xi})\big) - \mathcal{G}\big(T_n, T_{n+1}, \mathrm{U}^{k}_{n}(\boldsymbol{\xi})\big),
		\end{equation}
		where $k\ge 0$ is the iteration index and $n = 0,1,\dots,N-1$. 
		
		\begin{algorithm}[H]
			\caption{Parareal algorithm}
			\label{alg:parareal}
			\textbf{Initialization}: Generate initial guess $\{\mathrm{U}^0_n(\boldsymbol{\xi})\}_{n=1}^N$.\\
			\textbf{For} $k = 0,1,\cdots$ \\
			\textbf{Step 1}: On each subinterval $[T_n,T_{n+1}]$, compute
			$$\mathrm{\tilde{U}}_{n,j+1}^{k}(\boldsymbol{\xi}) = \mathcal{F}\big(t^k_{n,j}, t^k_{n,j+1}, \mathrm{\tilde{U}}_{n,j}^{k}(\boldsymbol{\xi})\big), \quad j=0,1,\cdots,J_n^k-1,$$\\
			with initial value $\mathrm{\tilde{U}}_{n,0}^{k}(\boldsymbol{\xi})=\mathrm{U}^k_n(\boldsymbol{\xi})$, where $\{t_{n,j}^k\}_{j=1}^{J_n^k-1}$ are the fine
			time points spaced arbitrarily within the large subinterval $[T_n,T_{n+1}]$.\\
			\textbf{Step 2}: Perform CGC\\
			$$\mathrm{U}_{n+1}^{k+1}(\boldsymbol{\xi}) = \mathcal{G}\big(T_n, T_{n+1}, \mathrm{U}_{n}^{k+1}(\boldsymbol{\xi})\big) + \mathrm{\tilde{U}}_{n,j+1}^{k}(\boldsymbol{\xi}) - \mathcal{G}\big(T_n, T_{n+1}, \mathrm{U}_{n}^{k}(\boldsymbol{\xi})\big), $$\\
			with $\mathrm{U}_{0}^{k+1}(\boldsymbol{\xi})=u_{0}(\boldsymbol{\xi})$.\\
			\textbf{Step 3}: If $\mathrm{U}_{n+1}^{k+1}(\boldsymbol{\xi})$ satisfies the stopping criterion, terminate the iteration;
			otherwise go to \textbf{Step 1}.
		\end{algorithm}

\subsection{ Diagonalization-based parallel CGC}

The $\mathcal{G}$-propagator enforces an inherently sequential CGC, limiting the parareal algorithm's speedup \cite{Y.M,J.M,J.M1}. 
Aiming at the bottleneck of sequential CGC in the parareal algorithm, a parallel CGC strategy is proposed. This approach introduces a parameter $\alpha$, which constrains the $\mathcal{G}$-propagator to act on the original system satisfying the coupling condition $u (0,\boldsymbol{\xi}) = \alpha u(T,\boldsymbol{\xi})$, and combines it with diagonalization techniques \cite{Y.E,S.L} to achieve parallelization of the CGC process \cite{F.K}.

We now generalize this strategy by employing the $\mathcal{G}$-propagator to solve a modified problem subject to a twisted boundary condition, as specified in:
            \begin{equation}
			\left\{ 
			\begin{aligned} 
				&u'(t,\boldsymbol{\xi}) = f \big(u (t, \boldsymbol{\xi} ), t, \boldsymbol{\xi} \big), \quad t \in [0, T], \\
				&u (0,\boldsymbol{\xi}) = \alpha u(T,\boldsymbol{\xi}).
			\end{aligned}
			\right.
			\label{eq:subprobs}
		\end{equation}
where $\alpha \in (0, 1)$ is a free parameter. With the same notation used for Algorithm \ref{alg:parareal}, we formulate the new parareal algorithm as shown in Algorithm \ref{alg:PGC}.
\begin{re}
    $\alpha$ is the core parameter of the convergence rate and diagonal rounding error of the regulation algorithm, and its selection should be combined with the problem type (linear/nonlinear), the spectral characteristics of the coefficient matrix and the propagation subtype.
    The selection of $\alpha$ centers on identifying a critical threshold \( \alpha^* \). When \( |\alpha| \leq \alpha^* \), the diagonalization-based parallel CGC achieves a convergence rate identical to that of the traditional sequential CGC. Moreover, at \( \alpha = \alpha^* \), the condition number of the diagonalization system remains \( \mathcal{O}(1) \)—i.e., independent of the time horizon \( T \)—while the rounding error is minimized.

In the linear case, \( \alpha^* \) depends on the spectral properties of the coefficient matrix \( A \) (including whether its eigenvalues are real or complex), as well as on the stability functions of both the \( \mathcal{F} \)- and \( \mathcal{G} \)-propagators. The key lies in matching the contraction factor \( \mathcal{K}_{\text{cla}} \) of the classical sequential CGC. For nonlinear problems, \( \alpha^* \) is influenced by the Lipschitz constant \( L \) and the time step size \( \Delta T \).

By contrast, under the conventional diagonalization setting (\( \alpha = 0 \)), the condition number grows rapidly with the number of time steps \( N \), leading to uncontrollable rounding errors. As a result, the algorithm diverges as \( T \) increases, rendering the choice \( \alpha = 0 \) impractical. There is a detailed
description of this parameter in \cite{S.L}.
\end{re}

\begin{algorithm}[h]
			\caption{Parareal algorithm with parallel CGC}
			\label{alg:PGC}
			\textbf{Initialization}: Generate initial guess $\{\mathrm{U}^0_n(\boldsymbol{\xi})\}_{n=1}^N$.\\
			\textbf{For} $k = 0,1,\cdots$ \\
			\textbf{Step 1}: On each subinterval $[T_n,T_{n+1}]$, compute
            \begin{align}\label{eq:PGC_step1}
            \mathrm{\tilde{U}}_{n,j+1}^{k}(\boldsymbol{\xi}) = \mathcal{F}\big(t^k_{n,j}, t^k_{n,j+1}, \mathrm{\tilde{U}}_{n,j}^{k}(\boldsymbol{\xi})\big), \quad j=0,1,\cdots,J_n^k-1,
            \end{align}
			with initial value 
            $$\mathrm{\tilde{U}}_{n,0}^{k}(\boldsymbol{\xi}) = 
            \begin{cases} 
            \mathrm{{U}}_{n}^{k}(\boldsymbol{\xi}), & n \geq 1, \\
            u_0(\boldsymbol{\xi}), & n = 0.
            \end{cases}
            $$
            
            $\mathrm{\tilde{U}}_{n,0}^{k}(\boldsymbol{\xi})=\mathrm{U}^k_n(\boldsymbol{\xi})$, where $\{t_{n,j}^k\}_{j=1}^{J_n^k-1}$ are the fine
			time points spaced arbitrarily within the large subinterval $[T_n,T_{n+1}]$.\\
			\textbf{Step 2}: Perform parallel CGC via diagonalization technique\\
			\begin{align}\label{eq:PGC_step2}
			    \mathrm{U}_{n+1}^{k+1}(\boldsymbol{\xi}) = \mathcal{G}\big(T_n, T_{n+1}, \mathrm{U}_{n}^{k+1}(\boldsymbol{\xi})\big) + \mathrm{\tilde{U}}_{n,j+1}^{k}(\boldsymbol{\xi}) - \mathcal{G}\big(T_n, T_{n+1}, \mathrm{U}_{n}^{k}(\boldsymbol{\xi})\big),
			\end{align} 
			with $\mathrm{U}_{0}^{k+1}(\boldsymbol{\xi})=\alpha\mathrm{U}_{N}^{k+1}(\boldsymbol{\xi})$ and $n=0,1,\cdots,N-1$.\\
			\textbf{Step 3}: If $\mathrm{U}_{n+1}^{k+1}(\boldsymbol{\xi})$ satisfies the stopping criterion, terminate the iteration;
			otherwise go to \textbf{Step 1}.
		\end{algorithm}

Note that in \eqref{eq:PGC_step1}, the $\mathcal{F}$-propagator uses $u_0(\boldsymbol{\xi})$ (rather than $\mathrm{U}^k_0(\boldsymbol{\xi})$) at $T_0 = 0$, ensuring the converged solution matches the correct numerical solution of \eqref{eq:subprobs}. We now detail the CGC procedure \eqref{eq:PGC_step2}. For any $k \geq 0$, let
\begin{align}
\boldsymbol{u}^k(\boldsymbol{\xi})=\begin{bmatrix}
\mathrm{U}_{1}^{k}(\boldsymbol{\xi})\\
\mathrm{U}_{2}^{k}(\boldsymbol{\xi})\\
\vdots\\
\mathrm{U}_{N}^{k}(\boldsymbol{\xi})
\end{bmatrix},
\boldsymbol{b}^k(\boldsymbol{\xi})=\begin{bmatrix}
b_{1}^{k}(\boldsymbol{\xi})\\
b_{2}^{k}(\boldsymbol{\xi})\\
\vdots\\
b_{N}^{k}(\boldsymbol{\xi})
\end{bmatrix}:=
\begin{bmatrix}
\mathrm{\tilde{U}}_{1,J_1^k}^k(\boldsymbol{\xi})-\mathcal{G}(T_0,T_1,\alpha \mathrm{U}_{N}^k(\boldsymbol{\xi}))\\
\mathrm{\tilde{U}}_{2,J_2^k}^k(\boldsymbol{\xi})-\mathcal{G}(T_1,T_2,\mathrm{U}_1^k(\boldsymbol{\xi}))\\
\vdots\\
\mathrm{\tilde{U}}_{N,J_{N}^k}^k(\boldsymbol{\xi})-\mathcal{G}(T_{N - 1},T_N,\mathrm{U}_{N - 1}^k(\boldsymbol{\xi}))
\end{bmatrix}.
\end{align}
Since we employ the backward-Euler method for the $\mathcal{G}$-propagator, the CGC procedure in \eqref{eq:PGC_step2} becomes
\begin{align}\label{eq:guodu}
\begin{cases}
\frac{s_1(\boldsymbol{\xi}) - \mathrm{U}_0^{k + 1}(\boldsymbol{\xi})}{\Delta T}=f(T_0,s_1(\boldsymbol{\xi})),&\mathrm{U}_1^{k + 1}(\boldsymbol{\xi})=s_1(\boldsymbol{\xi}) + b_1^k(\boldsymbol{\xi}),\\
\frac{s_2(\boldsymbol{\xi}) - \mathrm{U}_1^{k + 1}(\boldsymbol{\xi})}{\Delta T}=f(T_1,s_2(\boldsymbol{\xi})),&\mathrm{U}_2^{k + 1}(\boldsymbol{\xi})=s_2(\boldsymbol{\xi}) + b_2^k(\boldsymbol{\xi}),\\
\vdots\\
\frac{s_{N}(\boldsymbol{\xi})-\mathrm{U}_{N - 1}^{k + 1}(\boldsymbol{\xi})}{\Delta T}=f(T_{N - 1},s_{N}(\boldsymbol{\xi})),&\mathrm{U}_{N}^{k + 1}=s_{N}(\boldsymbol{\xi})+b_{N}^k(\boldsymbol{\xi}),
\end{cases}
\end{align}
where $\mathrm{U}_0^{k + 1}=\alpha \mathrm{U}_{N}^{k + 1}$ and $\{s_n\}_{n = 1}^{N}$ are the auxiliary variables. By eliminating the auxiliary variables in \eqref{eq:guodu} we have
\begin{align}
\begin{cases}
\frac{\mathrm{U}_1^{k + 1}(\boldsymbol{\xi})-\alpha \mathrm{U}_{N}^{k + 1}(\boldsymbol{\xi})}{\Delta T}f(T_0,\mathrm{U}_1^{k + 1}(\boldsymbol{\xi})-b_1^k(\boldsymbol{\xi}))+\frac{b_1^k(\boldsymbol{\xi})}{\Delta T},\\
\frac{\mathrm{U}_2^{k + 1}(\boldsymbol{\xi})-\mathrm{U}_1^{k + 1}(\boldsymbol{\xi})}{\Delta T}=f(T_1,\mathrm{U}_2^{k + 1}(\boldsymbol{\xi})-b_2^k(\boldsymbol{\xi}))+\frac{b_2^k(\boldsymbol{\xi})}{\Delta T},\\
\vdots\\
\frac{\mathrm{U}_{N}^{k + 1}(\boldsymbol{\xi})-\mathrm{U}_{N - 1}^{k + 1}(\boldsymbol{\xi})}{\Delta T}=f(T_{N_t - 1},\mathrm{U}_{N}^{k + 1}(\boldsymbol{\xi})-b_{N}^k(\boldsymbol{\xi}))+\frac{b_{N}^k(\boldsymbol{\xi})}{\Delta T},
\end{cases}
\end{align}
which is equivalent to
\begin{align}\label{eq:matrix_NT}
\left(\underbrace{\begin{bmatrix}
1&&&-\alpha\\
-1&1&&\\
&\ddots&\ddots&\\
&&-1&1
\end{bmatrix}}_{:=C_{\alpha}}\otimes I_x\right)\boldsymbol{u}^{k + 1}=\Delta T\underbrace{\begin{bmatrix}
f(T_0,\mathrm{U}_1^{k + 1}(\boldsymbol{\xi})-b_1^k(\boldsymbol{\xi}))\\
f(T_1,\mathrm{U}_2^{k + 1}(\boldsymbol{\xi})-b_2^k(\boldsymbol{\xi}))\\
\vdots\\
f(T_{N- 1},\mathrm{U}_{N}^{k + 1}(\boldsymbol{\xi})-b_{N}^k(\boldsymbol{\xi}))
\end{bmatrix}}_{:=\boldsymbol{F}(\boldsymbol{u}^{k + 1}(\boldsymbol{\xi}))}+\boldsymbol{b}^k(\boldsymbol{\xi}),
\end{align}
where $I_x\in\mathbb{R}^{N_x\times N_x}$ is an identity matrix.

Applying Newton's iteration to \eqref{eq:matrix_NT} yields:
\begin{align}\label{eq:Newton}
\boldsymbol{u}_{l + 1}^{k + 1}(\boldsymbol{\xi})=\boldsymbol{u}_l^{k + 1}(\boldsymbol{\xi})-(\boldsymbol{J}_l^{k + 1})^{-1}\left[(C_{\alpha}\otimes I_x)\boldsymbol{u}_l^{k + 1}(\boldsymbol{\xi})-\Delta T\boldsymbol{F}(\boldsymbol{u}_l^{k + 1}(\boldsymbol{\xi}))-\boldsymbol{b}^k(\boldsymbol{\xi})\right],
\end{align}
where $l = 0,1,\dots$ is the iteration index and $\boldsymbol{u}_l(\boldsymbol{\xi})=((\mathrm{U}_{1,l}^k(\boldsymbol{\xi}))^{\top},\dots,(\mathrm{U}_{N,l}^k(\boldsymbol{\xi}))^{\top})^{\top}$. The Jacobian matrix $\boldsymbol{J}_l^{k + 1}$ is given by
\begin{align*}
    &\boldsymbol{J}_l^{k + 1}(\boldsymbol{\xi})=C_{\alpha}\otimes I_x-\Delta T\text{blkdiag}(\nabla f_{1,l}^{k + 1}(\boldsymbol{\xi}),\dots,\nabla f_{N,l}^{k + 1}(\boldsymbol{\xi}))\\
\text{ with }\quad &\nabla f_{n,l}^{k + 1}(\boldsymbol{\xi}):=\nabla f(T_{n - 1},\mathrm{U}_{n,l}^{k + 1}(\boldsymbol{\xi})-b_n^k(\boldsymbol{\xi})).
\end{align*}

Following Gander and Halpern \cite{M.J2}, we approximate all Jacobian blocks by their average:
\begin{align}
\nabla f_{n,l}(\boldsymbol{\xi}) \approx \overline{\nabla f}_l^{k + 1}(\boldsymbol{\xi}) := \frac{1}{N}\sum_{n = 1}^{N} \nabla f(T_{n - 1},u_{n,l}^{k + 1}(\boldsymbol{\xi})-b_n^k(\boldsymbol{\xi})).
\end{align}
yielding the approximate Jacobian:
\begin{align}\label{eq:Jac}
\overline{\boldsymbol{J}}_l^{k + 1}(\boldsymbol{\xi}) = C_{\alpha} \otimes I_x - \Delta T I_t \otimes \overline{\nabla f}_l^{k + 1}(\boldsymbol{\xi}),
\end{align}
where $I_t \in \mathbb{R}^{N \times N}$ is the identity matrix. Substituting into \eqref{eq:Newton} gives the simplified Newton iteration:
\[
\boldsymbol{u}_{l + 1}^{k + 1}(\boldsymbol{\xi}) = \boldsymbol{u}_l^{k + 1}(\boldsymbol{\xi}) - (C_{\alpha} \otimes I_x - \Delta T I_t \otimes \overline{\nabla f}_l^{k + 1}(\boldsymbol{\xi}))^{-1}((C_{\alpha} \otimes I_x)\boldsymbol{u}_l^{k + 1} - \Delta T\boldsymbol{F}(\boldsymbol{u}_l^{k + 1}(\boldsymbol{\xi})) - \boldsymbol{b}^k(\boldsymbol{\xi})),
\]
which is equivalent to
\begin{align}\label{eq:Newton2}
\overline{\boldsymbol{J}}_l^{k + 1}\boldsymbol{u}_{l + 1}^{k + 1} = -\Delta T (I_t \otimes \overline{\nabla f}_l^{k + 1})\boldsymbol{u}_l^{k + 1} + \Delta T\boldsymbol{F}(\boldsymbol{u}_l^{k + 1}) + \boldsymbol{b}^k, \quad l = 0,1,\dots.
\end{align}

To solve the linear system \eqref{eq:Newton2} for $\boldsymbol{u}_{l+1}^{k+1}$, we exploit the $\alpha$-circulant structure of $C_{\alpha}$. From [1, Lemma 2.10], $C_{\alpha}$ admits the eigendecomposition: 
\begin{align}
C_{\alpha} = V \text{diag}(\lambda_1,\lambda_2,\dots,\lambda_{N})V^{-1}, \quad \lambda_n = 1 - \alpha^{\frac{1}{N}} \omega^{-(n - 1)},
\end{align}
where $\mathrm{i} = \sqrt{-1}$, $\omega = e^{\frac{2\pi \mathrm{i}}{N}}$, and the eigenvector matrix $V = \Lambda_{\alpha} \mathbb{F}_{N}$ combines:
\[
\Lambda_{\alpha} = \begin{bmatrix}
1 & & & \\
 & \alpha^{\frac{1}{N}} & & \\
 & & \ddots & \\
 & & & \alpha^{\frac{N - 1}{N}}
\end{bmatrix}, \quad
\mathbb{F}_{N} = \frac{1}{\sqrt{N}}
\begin{bmatrix}
1 & 1 & \dots & 1 \\
1 & \omega & \dots & \omega^{N - 1} \\
\vdots & \vdots & \dots & \vdots \\
1 & \omega^{N - 1} & \dots & \omega^{(N - 1)(N - 1)}
\end{bmatrix}.
\]
Using Kronecker product properties, we factorize the approximate Jacobian:
\[
\overline{\boldsymbol{J}}_l^{k + 1}(\boldsymbol{\xi}) = (V \otimes I_x) (\text{diag}(\lambda_1,\lambda_2,\dots,\lambda_{N_t}) \otimes I_x - \Delta T I_t \otimes \overline{\nabla f}_l^{k + 1}(\boldsymbol{\xi})) (V^{-1} \otimes I_x),
\]
this factorization enables solving for $\boldsymbol{u}_{l+1}^{k+1}$ through three computational steps:
\begin{align}\label{eq:three steps}
\begin{cases}
\boldsymbol{p} = (V^{-1} \otimes I_x)\boldsymbol{r}^k = (\mathbb{F}_{N_t}^* \otimes I_x)[(\Lambda_{\alpha}^{-1} \otimes I_x)\boldsymbol{r}^k], & \text{Step-(a)}, \\
(\lambda_n I_x - \Delta T \overline{\nabla f}_l^{k + 1}) q_n = p_n, \quad n = 1,2,\dots,N_t, & \text{Step-(b)}, \\
\boldsymbol{u}_{l + 1}^{k + 1} = (V \otimes I_x)\boldsymbol{q} = (\Lambda_{\alpha} \otimes I_x)[(\mathbb{F}_{N_t} \otimes I_x)\boldsymbol{q}], & \text{Step-(c)},
\end{cases}
\end{align}
where $\boldsymbol{r}^k = -\Delta T (I_t \otimes \overline{\nabla f}_l^{k + 1})\boldsymbol{u}_l^{k + 1} + \Delta T\boldsymbol{F}(\boldsymbol{u}_l^{k + 1}) + \boldsymbol{b}^k$, $\boldsymbol{p} = (p_1^{\top},p_2^{\top},\dots,p_{N}^{\top})^{\top}$ and $\boldsymbol{q} = (q_1^{\top},q_2^{\top},\dots,q_{N}^{\top})^{\top}$. 
In \eqref{eq:three steps}, the first and third steps involve only matrix-vector multiplications, which can be efficiently implemented via the Fast Fourier Transform (FFT) due to the structure of the discrete Fourier matrix $\mathbb{F}_{N}$. The dominant computational cost lies in the second step, which, notably, is inherently parallel. Furthermore, this diagonalization-based CGC requires no additional processors or storage compared to the sequential CGC.

The method solves the linear system derived from the backward Euler discretization by exploiting the circulant matrix structure, allowing parallel computation through FFT. As outlined in the original work, parallelism is achieved using Kronecker products and eigenvector decompositions, eliminating storage overhead.

\subsection{The implementation of the KLE-PCGC method}
In this subsection, we will introduce the framework of KLE-PCGC method for solving time-dependent problems with random inputs. Although the diagonalization-based CGC method has reduced
computational effort, it still be repeatedly computed with random initial value for different parameters. For diagonalization-based CGC  algorithms, the number of iteration steps is influenced by the magnitude of the initial error (for given tolerance $\epsilon$,
the anticipated number of iterations is $k = \frac{\log(\epsilon/\|\mathbf{e}^0\|}{\log \rho}$, where $\mathbf{e}^0 = (e^0_1, e^0_2, \cdots , e^0_N)^T$ denotes the initial error and $\rho$ denotes convergence factor). To further enhance computational efficiency, we aim to achieve accurate initial value prediction for systems with randomly input parameters by combining the Algorithm \ref{alg:PGC}
with KL expansion, see Algorithm \ref{alg:KLE-PCGC} below.

\begin{algorithm}[H]
			\caption{KL expasion with parallel CGC algorithm}
			\label{alg:KLE-PCGC}
            \textbf{Initialization}: Generate initial value by KL expansion $\{\mathrm{U}^0_n(\boldsymbol{\xi})\}_{n=1}^N$.\\
             ~~~~\textbf{Input}: A training set $\Xi_{t} \subset \Xi$ and a tolerance $\varepsilon^\text{KL}$\\
     ~~~~\textbf{Output}: Variable-separation representation $\mathrm{U}^0(x,t_n,\boldsymbol{\xi})\approx \sum\limits_{i=1}^{M_Q} c_i \zeta_i(\boldsymbol{\xi})\mathrm{U}^0_i(x,t_n)$\\
      ~~~~1:~~Compute the snapshots $\{\mathrm{U}^0(x,t_n,\boldsymbol{\xi}_i)\}_{i=1}^{n_{n_t}}$ by (\ref{eq:psys}) for all $\boldsymbol{\xi}_i\in \Xi_{t}$ and construct \\
      $~~~~~$the covariance
       matrix \textbf{C};\\
      ~~~~2:~~Solve the eigenvalue problem and determine $M_Q$ such that $\frac{\sum _{i=1}^{M_Q}\hat{\lambda}_i}{\sum _{i=1}^{n_t}\hat{\lambda}_i}\leq1-\varepsilon^\text{KL};$\\
      ~~~~3:~~Assemble $\textbf{B}$ based on gPC basis functions and $\Xi_{t}$ by (\ref{least square A});\\
      ~~~~4:~~Construct the functions $\{g_i(x,t_n)\}_{i=1}^{M_Q}$ by (\ref{KLE basis}), for each $i=1,..,M_Q$, assemble\\ $~~~~~$\textbf{D} by (\ref{least square F});\\
      ~~~~5:~~For each $i=1,..,M_Q$, solve problem (\ref{least square problem}) by least square procedure to obtain\\ $~~~~~$$\textbf{h}$ and then get $\zeta_i(\boldsymbol{\xi})\approx\sum\limits_{i=1}^{P}h_i\psi_i(\boldsymbol{\xi})$;\\
      ~~~~6:~~Return the representation \\ $~~~~~$$\mathrm{U}^0(x,t_n,\boldsymbol{\xi})\approx \sum\limits_{i=1}^{M_Q} \sqrt{\hat{\lambda}_i} \zeta_i(\boldsymbol{\xi})g_i(x,t_n)\approx\sum\limits_{i=1}^{M_Q}\sum\limits_{j=1}^{P} \sqrt{\hat{\lambda}_i}h_j\psi_j(\boldsymbol{\xi}) g_i(x,t_n)$.\\
			
			\textbf{For} $k = 0,1,\cdots$ \\
			\textbf{Step 1}: On each subinterval $[T_n,T_{n+1}]$, compute
            \begin{align}\label{eq:PGC_step1}
            \mathrm{\tilde{U}}_{n,j+1}^{k}(\boldsymbol{\xi}) = \mathcal{F}\big(t^k_{n,j}, t^k_{n,j+1}, \mathrm{\tilde{U}}_{n,j}^{k}(\boldsymbol{\xi})\big), \quad j=0,1,\cdots,J_n^k-1,
            \end{align}
			with initial value 
            $$\mathrm{\tilde{U}}_{n,0}^{k}(\boldsymbol{\xi}) = 
            \begin{cases} 
            \mathrm{{U}}_{n}^{k}(\boldsymbol{\xi}), & n \geq 1, \\
            u_0(\boldsymbol{\xi}), & n = 0.
            \end{cases}
            $$
            
            $\mathrm{\tilde{U}}_{n,0}^{k}(\boldsymbol{\xi})=\mathrm{U}^k_n(\boldsymbol{\xi})$, where $\{t_{n,j}^k\}_{j=1}^{J_n^k-1}$ are the fine
			time points spaced arbitrarily within the large subinterval $[T_n,T_{n+1}]$.\\
			\textbf{Step 2}: Perform parallel CGC via diagonalization technique\\
			\begin{align}\label{eq:PGC_step2}
			    \mathrm{U}_{n+1}^{k+1}(\boldsymbol{\xi}) = \mathcal{G}\big(T_n, T_{n+1}, \mathrm{U}_{n}^{k+1}(\boldsymbol{\xi})\big) + \mathrm{\tilde{U}}_{n,j+1}^{k}(\boldsymbol{\xi}) - \mathcal{G}\big(T_n, T_{n+1}, \mathrm{U}_{n}^{k}(\boldsymbol{\xi})\big),
			\end{align} 
			with $\mathrm{U}_{0}^{k+1}(\boldsymbol{\xi})=\alpha\mathrm{U}_{N}^{k+1}(\boldsymbol{\xi})$ and $n=0,1,\cdots,N-1$.\\
			\textbf{Step 3}: If $\mathrm{U}_{n+1}^{k+1}(\boldsymbol{\xi})$ satisfies the stopping criterion, terminate the iteration;
			otherwise go to \textbf{Step 1}.
		\end{algorithm}

\section{Convergence analysis}\label{cov}
In this section, we analyze the convergence properties and computational speedup for parameterized dynamic problems described by Eq.~\eqref{eq:psys}, employing a parareal algorithm with diagonalized CGC. Our analysis builds upon the foundational work of  upon foundational results established by Wu \cite{S.L}, who established convergence results for parameter-independent ordinary differential equation (ODE) system.

\subsection{The linear systems}
For linear ODE system
\begin{equation}
		\left\{ 
		\begin{aligned}
			u'(t,\boldsymbol{\xi})+A(\boldsymbol{\xi})u(t,\boldsymbol{\xi})&= f(\boldsymbol{\xi}),
			\\
			u(0)&= u_0(\boldsymbol{\xi}),
		\end{aligned}
		\right.
		\label{eq:ODE}
	\end{equation}
where $A(\boldsymbol{\xi})=A\in \mathbb{R}^{N_x\times N_x}$ and the following slightly ``wrong'' model:
\begin{equation}
		\left\{ 
		\begin{aligned}
			u'(t,\boldsymbol{\xi})+A(\boldsymbol{\xi})u(t,\boldsymbol{\xi})&= f(\boldsymbol{\xi}),
			\\
			u(0,\boldsymbol{\xi})&= \alpha u(T,\boldsymbol{\xi}),
		\end{aligned}
		\right.
		\label{eq:WODE}
	\end{equation}

\begin{lemma}[general result deduced from \cite{M.J}]\label{lem1}
Let $\mathcal{F}$ and $\mathcal{G}$ be two one-step numerical methods with stability functions $\mathcal{R}_f(z)$ and $\mathcal{R}_g(z)$, which are, respectively, applied to the ODE system \eqref{eq:ODE} and \eqref{eq:WODE} with small step size $\Delta t$ and large step size $\Delta T$. Then, the error $\mathbf{e}^k(\boldsymbol{\xi}):= (e_1^k(\boldsymbol{\xi}), e_2^k(\boldsymbol{\xi}), \ldots, e_{N}^k(\boldsymbol{\xi}))^\top$ satisfies
\begin{equation}\label{eq:er}
\|\mathbf{e}^{k+1}(\boldsymbol{\xi})\| \leq \|\mathbf{G}^{-1}(\mathbf{G} - \mathbf{F})\| \|\mathbf{e}^k(\boldsymbol{\xi})\|,
\end{equation}
where $\|\boldsymbol{\cdot}\|$ is an arbitrary norm and the matrices $\mathbf{G}$ and $\mathbf{F}$ are given by
\[
\mathbf{G} = \begin{pmatrix}
I_x & & & & -\alpha \mathcal{R}_g(\Delta T A) \\
-\mathcal{R}_g(\Delta T A) & I_x & & & \\
0 & -\mathcal{R}_g(\Delta T A) & I_x & & \\
\vdots & \ddots & \ddots & \ddots & \\
0 & \ldots & 0 & -\mathcal{R}_g(\Delta T A) & I_x
\end{pmatrix},
\]
\[
\mathbf{F} = \begin{pmatrix}
I_x & & & \\
-\mathcal{R}_f^J(\Delta t A) & I_x & & \\
0 & -\mathcal{R}_f^J(\Delta t A) & I_x & \\
\vdots & \ddots & \ddots & \ddots \\
0 & \ldots & 0 & -\mathcal{R}_f^J(\Delta t A) & I_x
\end{pmatrix}.
\]
\end{lemma}

Let $A = V_A D_A V_A^{-1}$ with $D_A = \text{diag}(\mu_1, \mu_2, \ldots, \mu_m)$ and $V_A$ consisting of the eigenvectors of $A$. Define the norm $\|\boldsymbol{\cdot}\|_{\infty}$ via the $\infty$-norm:
\begin{equation}\label{eq:norm}
\tnormtight[\infty]{\boldsymbol{u}} := \|(I_t \otimes V_A)\boldsymbol{u}\|_{\infty} \quad \forall \boldsymbol{u} \in \mathbb{R}^{N_x N_t}.
\end{equation}

Then, for any matrix $\mathbf{M} \in \mathbb{R}^{N_x N_t \times N_x N_t}$ the induced matrix norm is
\[
\tnormtight[\infty]{\mathbf{M}} = \|(I_t \otimes V_A)\mathbf{M} (I_t \otimes V_A^{-1})\|_{\infty}.
\]

Let $\|\boldsymbol{\cdot}\| = \tnormtight[\infty]{\cdot}$ in \eqref{eq:er}. Then, we have
\begin{equation}
\tnormtight[\infty]{\mathbf{G}^{-1}(\mathbf{G} - \mathbf{F})} \leq \max_{z \in \sigma(\Delta T A)} \|G^{-1}(z) \left(G(z) - F\left(\frac{z}{J}\right)\right)\|_{\infty},
\end{equation}
where $\sigma(\Delta T A)$ denotes the spectrum of the matrix $\Delta T A$ and $G(z), F(z) \in \mathbb{R}^{N_t \times N_t}$ are given by
\[
G(z) = \begin{pmatrix}
1 & & & & -\alpha \mathcal{R}_g(z) \\
-\mathcal{R}_g(z) & 1 & & & \\
0 & -\mathcal{R}_g(z) & 1 & & \\
\vdots & \ddots & \ddots & \ddots & \\
0 & \ldots & 0 & -\mathcal{R}_g(z) & 1
\end{pmatrix},
\]
\[
F(z) = \begin{pmatrix}
1 & & & \\
-\mathcal{R}_f^J(z) & 1 & & \\
0 & -\mathcal{R}_f^J(z) & 1 & \\
\vdots & \ddots & \ddots & \ddots \\
0 & \ldots & 0 & -\mathcal{R}_f^J(z) & 1
\end{pmatrix}.
\]

\begin{theorem}\label{linearthm}
Let $\mathcal{G}$ be the backward-Euler method with step size $\Delta T$ and let the coefficient matrix $A$ in \eqref{eq:ODE} be stable (i.e., all the eigenvalues have positive real parts). Then for the parareal algorithm with diagonalization-based CGC it holds that

\begin{equation}\label{eq:er1}
\tnormtight[\infty]{\mathbf{e}^{k+1}(\boldsymbol{\xi})} \leq \max_{z \in \sigma(\Delta T A)} \mathcal{K}(z, J, \alpha) \tnormtight[\infty]{\mathbf{e}^k(\boldsymbol{\xi})},
\end{equation}
where the norm $\tnormtight[\infty]{\boldsymbol{\cdot}}$ is defined by \eqref{eq:norm}. The quantity $\mathcal{K}$, which we call the convergence factor corresponding to a single eigenvalue (or in short ``contraction factor'' hereafter), is given by
\begin{equation}\label{eq:er1_condition}
\begin{split}
&\mathcal{K}(z, J, \alpha) = \max\left\{|\alpha \mathcal{R}_g(z)|(1 + \mathcal{K}_{\text{cla}}(z, J)), \mathcal{K}_{\text{cla}}(z, J)\right\}, \\
&with\quad\mathcal{K}_{\text{cla}}(z, J) := \frac{\left|\mathcal{R}_f^J\left(\frac{z}{J}\right) - \mathcal{R}_g(z)\right|}{1 - |\mathcal{R}_g(z)|}.
\end{split}
\end{equation}
The notation in \eqref{eq:er1}-\eqref{eq:er1_condition} are the same as those appearing in Lemma \ref{lem1}.
\end{theorem}
\begin{proof}
For the parameter-independent case, the proof is given in  
\cite[section 3]{S.L}. Without essential change, the proof can be extended to the parameter-dependent system case here. We therefore omit the details.
\end{proof}

\begin{re}
The function $\mathcal{K}_{\text{cla}}(z, J)$ is the contraction factor of the parareal algorithm with classical CGC; see, e.g., \cite{M.J,S.L1,S.L2}.
\end{re}

By combining Theorem \ref{linearthm} with the estimates given in (\ref{eq:er1}) and the (\ref{output-approximation}), we can derive the following convergence result.
\begin{theorem}
	Assume for some small $\delta > 0$. Let $u(T_n;\boldsymbol{\xi})$ be the solution to \eqref{eq:ODE} and $\mathrm{U}^{k+1}_{n}(\boldsymbol{\xi})$ be the $(k+1)$th parallel CGC iteration numerical approximation of  \eqref{eq:WODE}. There exist a constant $C$ dependent on $\sigma(\Delta TA)$, $J$, and $\alpha$ such that
	\begin{equation}
		\tnormtight[\infty]{u(T_n;\boldsymbol{\xi})-\mathrm{U}^{k+1}_n(\boldsymbol{\xi})}
		\le C^{k+1}\tnormtight[\infty]{\delta},
	\end{equation}
    where $C= \max\limits_{z\in \sigma(\Delta TA)}\mathcal{K}(z, J,\alpha)$.
\end{theorem}
\begin{proof}
    \begin{align*}	\tnormtight[\infty]{u(T_n;\boldsymbol{\xi})-\mathrm{U}^{k+1}_n(\boldsymbol{\xi})}
		&\le  \max_{z\in \sigma(\Delta TA)}\mathcal{K}(z, J,\alpha)\tnormtight[\infty]{\mathbf{e}^k(\boldsymbol{\xi})}\\
        &\le \big(\max_{z\in \sigma(\Delta TA)}\mathcal{K}(z, J,\alpha)\big)^{k+1}\tnormtight[\infty]{\mathbf{e}^0(\boldsymbol{\xi})}\\
        &\le \big(\max_{z\in \sigma(\Delta TA)}\mathcal{K}(z, J,\alpha)\big)^{k+1}\tnormtight[\infty]{\delta}
	\end{align*}
\end{proof}

\subsection{The nonlinear systems}

For nonlinear system:
\begin{equation}
		\left\{ 
		\begin{aligned}
			u'(t,\boldsymbol{\xi})&= f(u(t,\boldsymbol{\xi}),t,\boldsymbol{\xi}),
			\\
			u(0,\boldsymbol{\xi})&= u_0(\boldsymbol{\xi}),
		\end{aligned}
		\right.
		\label{eq:non_ODE}
	\end{equation}
where $f: \mathbb{R}^{N_x}\times \mathbb{R}^{+}\to \mathbb{R}^{N_x}$. Following the linear case treatment, the $\mathcal{F}$-propagator operates on system \eqref{eq:non_ODE}, whereas the $\mathcal{G}$-propagator acts on the modified problem:
\begin{equation}
		\left\{ 
		\begin{aligned}
			u'(t,\boldsymbol{\xi})&= f(u(t,\boldsymbol{\xi}),t,\boldsymbol{\xi}),
			\\
			u(0,\boldsymbol{\xi})&= \alpha u(T,\boldsymbol{\xi}),
		\end{aligned}
		\right.
		\label{eq:non_WODE}
	\end{equation}
\begin{assumption}\label{assump}
For the function $f\big(u(t,\boldsymbol{\xi}),t,\boldsymbol{\xi}\big)$ appearing in \eqref{eq:non_ODE} suppose there exists a constant $L > 0$ such that the following one-sided Lipschitz condition holds:
\begin{equation}\label{eq:Lip}
\big\langle f\big(u_1(t,\boldsymbol{\xi}), t,\boldsymbol{\xi}\big) - f\big(u_2(t,\boldsymbol{\xi}), t,\boldsymbol{\xi}\big), u_1(t,\boldsymbol{\xi}) - u_2(t,\boldsymbol{\xi})\big \rangle \leq -L \|u_1(t,\boldsymbol{\xi}) - u_2(t,\boldsymbol{\xi})\|_2,
\end{equation}
where $\langle \cdot \rangle$ denotes the Euclid inner product. Moreover, we assume that the $\mathcal{F}$-propagator is an exact solver.
\end{assumption}

\begin{lemma}\label{lem3}
Under Assumption \ref{assump}, we get for the $\mathcal{F}$-propagator
\begin{equation}\label{eq:Lip_F}
\|\mathcal{F}(T_n, T_{n+1}, u_1(\boldsymbol{\xi}) - \mathcal{F}(T_n, T_{n+1}, u_2(\boldsymbol{\xi})\|_2 \leq e^{-L \Delta T} \|u_1(\boldsymbol{\xi}) - u_2(\boldsymbol{\xi})\|_2 \quad \forall u_1, u_2 \in \mathbb{R}^{N_x}.
\end{equation}
For the $\mathcal{G}$-propagator it holds that
\begin{equation}\label{eq:Lip_G}
\|\mathcal{G}(T_n, T_{n+1}, u_1(\boldsymbol{\xi}) - \mathcal{G}(T_n, T_{n+1}, u_2(\boldsymbol{\xi})\|_2 \leq \frac{1}{1 + \Delta t L} \|u_1(\boldsymbol{\xi} )- u_2(\boldsymbol{\xi})\|_2 \quad \forall u_1, u_2 \in \mathbb{R}^{N_x}.
\end{equation}
\end{lemma}

\begin{proof}
For the scalar case $f : \mathbb{R} \times \mathbb{R}^+ \to \mathbb{R}$, the proof is given in [section 5]\cite{M.J1}. Without essential change, the proof can be extended to the system case $f : \mathbb{R}^{N_x} \times \mathbb{R}^+ \to \mathbb{R}^{N_x}$ here. We therefore omit the details.
\end{proof}

\begin{theorem}\label{thm2}
Under Assumption \ref{assump}, the errors of Algorithm \ref{alg:PGC}, i.e., the parareal algorithm with diagonalization-based CGC, satisfy
\begin{equation}\label{eq:non_er}
\max_{n=1,2,\ldots,N} \|u(T_n,\boldsymbol{\xi}) - \mathrm{U}_n^{k+1}(\boldsymbol{\xi})\|_2 \leq \rho(\alpha) \max_{n=1,2,\ldots,N} \|u(T_n,\boldsymbol{\xi}) - \mathrm{U}_n^k(\boldsymbol{\xi})\|_2 \quad \forall k \geq 0
\end{equation}
with
\[
\rho(\alpha) = \max\left\{ \left| \alpha \right| \frac{1 + e^{-L \Delta T}}{L \Delta T}, \frac{e^{-L \Delta T} + \frac{1}{1 + L \Delta T}}{1 - \frac{1}{1 + L \Delta T}} \right\},
\]
provided $\frac{|\alpha|}{1 + L \Delta T} < 1$ and $N \gg 1$.
\end{theorem}
\begin{proof}
For the parameter-independent case, the proof is given in  
\cite[section 4]{S.L}. Without essential change, the proof can be extended to the parameter-dependent system case here. We therefore omit the details.
\end{proof}

\begin{re} 
If $\alpha = 0$, the quantity $\rho$ reduces to $\rho_{\text{cla}} := \frac{e^{-L \Delta T} + \frac{1}{1 + L \Delta T}}{1 - \frac{1}{1 + L \Delta T}}$, which is the result given by Gander et al. for the so-called PP-PC algorithm applied to nonlinear scalar time-periodic differential equations; see \cite{M.J1}, Theorem \ref{thm2} for more details.
\end{re}

By combining Theorem \ref{thm2} with the estimates given in (\ref{eq:non_er}) and the (\ref{output-approximation}), we can also derive the following convergence result.
\begin{theorem}
    Assume for some small $\delta > 0$. Let $u(T_n;\boldsymbol{\xi})$ be the solution to \eqref{eq:non_ODE} and $\mathrm{U}^{k+1}_{n}(\boldsymbol{\xi})$ be the $(k+1)$th parallel CGC iteration numerical approximation of  \eqref{eq:non_WODE}. There exist a constant $\rho(\alpha)$ dependent on Lipschitz constant $L$, $\Delta$, and $\alpha$ such that
	\begin{equation}
		\max_{n=1,2,\cdot,N}\|u(T_n;\boldsymbol{\xi})-\mathrm{U}^{k+1}_n(\boldsymbol{\xi})\|_2
		\le \rho^{k+1}(\alpha)\delta.
	\end{equation}
\end{theorem}

\section{Numerical results}\label{nu}
In this section, we provide three numerical examples to illustrate the applicability of the proposed KLE-PCGC method for solving time-dependent problems with random inputs. The results demonstrate that KLE-PCGC not only achieves higher computational efficiency compared to the classical parareal algorithm, but also exhibits significant speedup over diagonalization-based parallel CGC when applied to such problems. In all experiments, the standard parareal and diagonalization-based parallel CGC start from a random initial guess and terminate when the error falls below $10^{-10}$.

Let $u_n(\boldsymbol{\xi})$ denote the reference solution, $u_n^k(\boldsymbol{\xi})$ denote the solution obtained by the standard parareal method, $\tilde{u}_n^k(\boldsymbol{\xi})$ denote the solution obtained by the diagonalization-based parallel CGC method, and $\mathrm{U}_n^k$ denote the solution from the KLE-PCGC method. The errors in the mean sense are defined respectively as:

\begin{equation}
\bm{err}_k = (e_1^k, e_2^k, \cdots, e_{N}^k)^{\top}, \quad
e_n^k = \frac{1}{S} \sum_{i=1}^{S}\| u_n(\boldsymbol{\xi}_i) - u_n^k(\boldsymbol{\xi}_i) \|_{\infty}, \quad n = 1, 2, \cdots, N,
\end{equation}

\begin{equation}
\bar{\bm{err}}_k = (\bar{e}_1^k, \bar{e}_2^k, \cdots, \bar{e}_{N}^k)^{\top}, \quad
\bar{e}_n^k = \frac{1}{S} \sum_{i=1}^{S}\| u_n(\boldsymbol{\xi}_i) - \tilde{u}_n^k(\boldsymbol{\xi}_i) \|_{\infty}, \quad n = 1, 2, \cdots, N,
\end{equation}
and

\begin{equation}
\tilde{\bm{err}}_k = (\tilde{e}_1^k, \tilde{e}_2^k, \cdots, \tilde{e}_{N}^k)^{\top}, \quad
\tilde{e}_n^k = \frac{1}{S} \sum_{i=1}^{S} \| u_n(\boldsymbol{\xi}_i) - \mathrm{U}_n^k(\boldsymbol{\xi}_i) \|_{\infty}, \quad n = 1, 2, \cdots, N.
\end{equation}

\subsection{Advection-diffusion equation}\label {ADE}
\par We now consider a 2D advection-diffusion equation given by:
\begin{align*}
	\begin{cases}
		\frac{\partial u(x,t;\boldsymbol{\xi})}{\partial t}-\nabla(a(x,t;\boldsymbol{\xi})\nabla u(x,t;\boldsymbol{\xi}))+b(x,t
        )\cdot\nabla u(x,t;\boldsymbol{\xi})=f(x,t;\boldsymbol{\xi}),\quad &(x,t)\in \Omega\times[0,T],\\
		u(x,t;\boldsymbol{\xi})=0,\quad& (x,t)\in \partial\Omega\times[0,T],\\
		u(x,0;\boldsymbol{\xi})=u_0,\quad &x\in \Omega,
	\end{cases}
\end{align*}
where the coefficient is given by $a(x,t;\boldsymbol{\xi})=0.5(2+\cos(\pi\xi^1)^2)$, $b(x,t)=0.1\sin(\frac{\pi}{2})(x_2,x_1)^T$ with $\xi^1 \sim U [2,6]$.
The spatial domain is $\Omega=(0,1)^2$, and final time $T=1$. The source term $f(x,t;\boldsymbol{\xi})$ is defined as:
\begin{align*}
    f(x,t;\boldsymbol{\xi})=&\exp(-t)\big(-\sin(\pi x_1)\sin(2\pi x_2)+0.5\pi^2(2+\cos(\pi\xi^1)^2)\sin(\pi x_1)\sin(2\pi x_2)+\\&0.05\pi\sin(\pi/2)x_2\cos(\pi x_1)\sin(2\pi x_2)+0.1\pi\sin(\pi/2)x_1\sin(\pi x_1)\cos(2\pi x_2)\big),
\end{align*}
with the initial value $u_0(x)=\sin(\pi x_1)\cdot sin(2\pi x_2)$.

For spatial discretization, P1-Lagrange finite elements are used with grid size $\Delta x_1 = \Delta x_2 = h$. For time discretization, the backward Euler method is applied with a coarse time step fixed at $\Delta T = 1 / 24$, and the spatial mesh size is fixed at $h = 1/20$. The $\mathcal{F}$ propagator also uses the backward Euler method. The reference solution $\{ u_n\}^{N}_{n=1}$ (at the coarse time points) is computed sequentially using a uniform fine time step $\Delta t = \Delta T / 50$. The training set for generating initial values at the coarse time points has size $|\Xi_t|=10$, and the tolerance for the Karhunen-Lo\`{e}ve expansion is $\varepsilon^{KL}=10^{-10}$.

\begin{figure}[htbp]
	\centering
	\includegraphics[height=6cm,width=6cm]{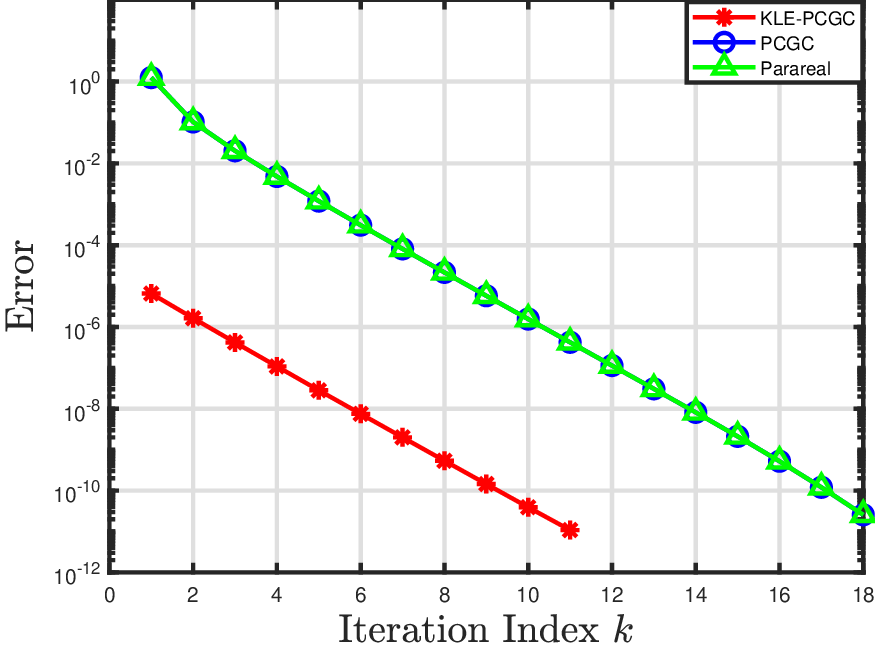}
	
	\caption{Comparison of the mean error corresponding to the parareal, parallel CGC and the  KLE-PCGC algorithm.}\label{ex3}
\end{figure}

In this experiment, the parameter 
$\alpha=0.1$ was chosen, and numerical tests with 1000 random parameter samples were conducted to assess the computational efficiency of the proposed KLE-PCGC method. Fig. \ref{ex3} illustrates the evolution of the error as a function of the iteration count for the three algorithms. It is evident that the error of the KLE-PCGC method decays at a notably faster rate compared to both the standard Parareal and the diagonalization based parallel CGC algorithms. From the initial iterations, KLE-PCGC exhibits a sharper error reduction and attains a lower error level more rapidly as the iteration index $k$ increases. In particular, for the same number of iterations, the error of KLE-PCGC remains substantially smaller than those of the other two methods. The KLE-PCGC algorithm satisfies the preset stopping tolerance by $k=10$, whereas the standard parareal and diagonalization-based parallel CGC require 18 iterations to converge. These results confirm that KLE-PCGC achieves more effective numerical error reduction and yields a more efficient approximation to the reference solution. Similarly, the error shows that KLE-PCGC, parareal, and the diagonalization based parallel CGC algorithms share an identical convergence rate, confirming the theoretical results.

Fig. \ref{ex3_1} shows a comparison of the spatial distribution of the mean reference solution \( u_N(\boldsymbol{\xi}) \), the solution \( U_N^k(\boldsymbol{\xi}) \) by KLE-PCGC method, the solution $\tilde{u}_N^k(\boldsymbol{\xi})$ by diagonalization-based parallel CGC method and the solution \( u_N^k(\boldsymbol{\xi}) \) by parareal method at the final time $t=T$. It can be observed that all solutions are in high agreement in terms of both global morphology and local features, with both the KLE-PCGC, diagonalization-based parallel CGC and parareal solutions capturing the spatial structure of the reference solution well.

\subsection{Burgers Equation}
We now consider the one-dimensional nonlinear Burgers equation:
\begin{align*}
	\begin{cases}
		\frac{\partial u(x,t)}{\partial t} + u(x,t) \cdot \frac{\partial u(x,t)}{\partial x} = \frac{\epsilon}{50} \cdot \frac{\partial^2 u(x,t)}{\partial x^2},\quad &(x,t)\in \Omega\times[0,T],\\
		u(x,t)=0,\quad& (x,t)\in \partial\Omega\times[0,T],\\
		u(x,0)=u_0,\quad &x\in \Omega,
	\end{cases}
\end{align*}
where $\epsilon \sim U[1,3]$. The spatial domain is $\Omega=(0,1)$ and the final time is $T=2$.

\begin{figure}[H]
	\centering
	\includegraphics[height=6cm,width=6cm]{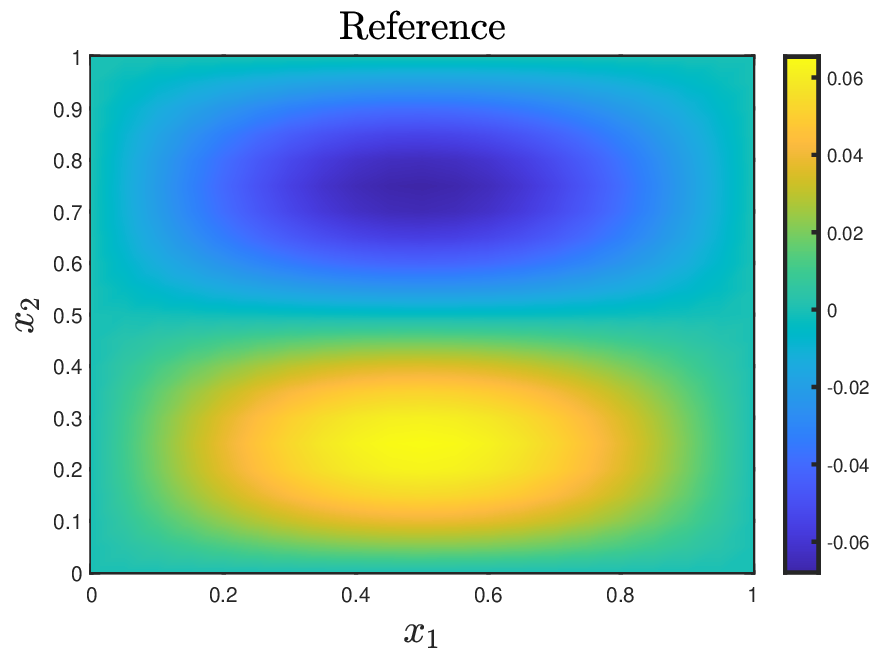}
	\qquad\qquad
	\includegraphics[height=6cm,width=6cm]{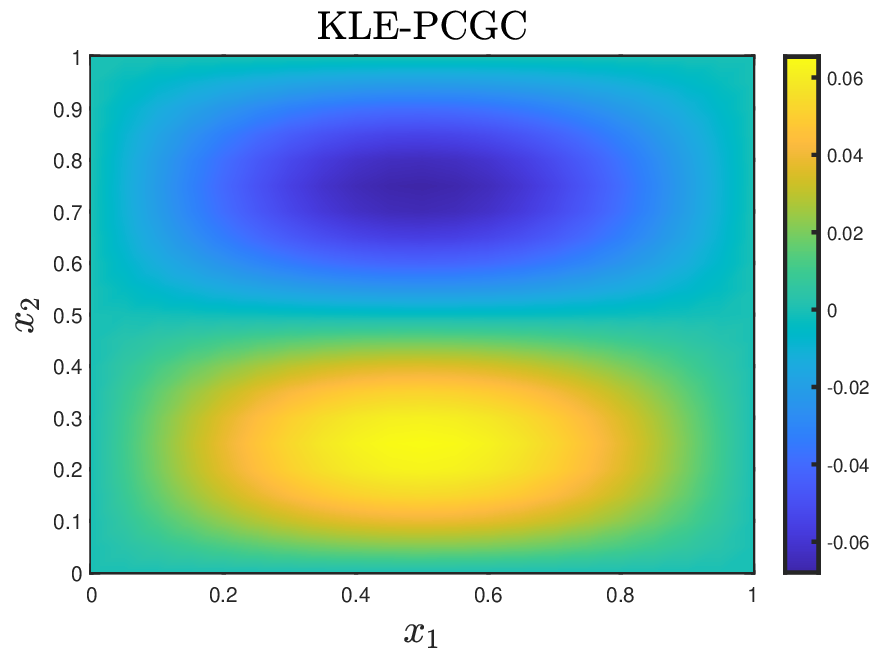}\\
    \includegraphics[height=6cm,width=6cm]{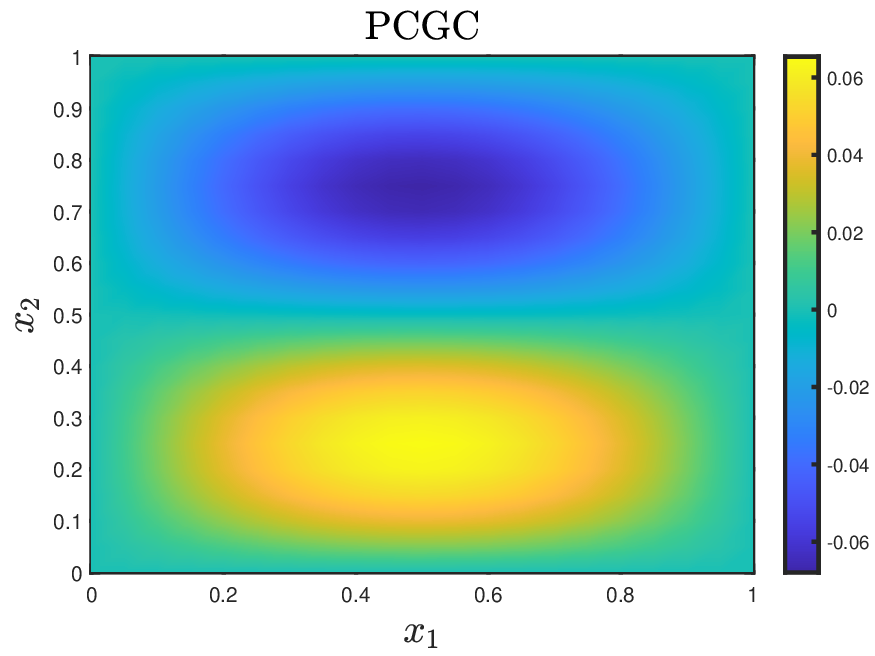}
	\qquad\qquad
	\includegraphics[height=6cm,width=6cm]{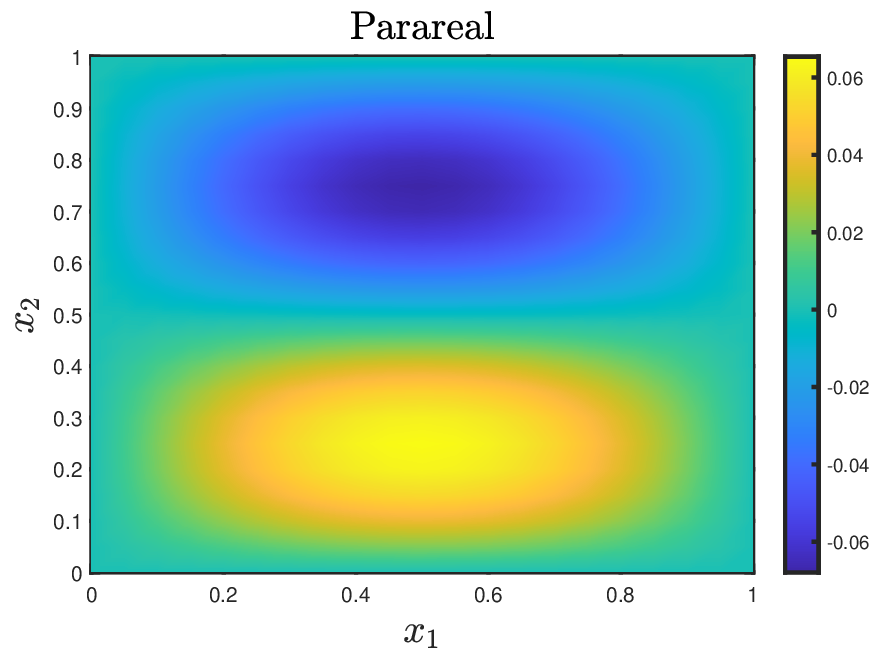}
	\caption{The mean reference solution $u_N(\boldsymbol{\xi})$ (left), the KLE-PCGC solution $\mathrm{U}_N^k(\boldsymbol{\xi})$ (middle), and the Parareal solution $u_N^k(\boldsymbol{\xi})$ (right) at the final time $t=T$.}\label{ex3_1}
\end{figure}

For spatial discretization, an upwind scheme is used for the convective term and central differencing for the diffusive term, with grid size $\Delta x = 1/100$. For time discretization, the backward Euler method is employed with a coarse time step fixed at $\Delta T = 2 / 25$; the $\mathcal{F}$ propagator also uses backward Euler. The reference solution $\{ u_n\}^{N}_{n=1}$ (at the coarse time points) is computed sequentially using a uniform fine time step $\Delta t = \Delta T / 40$. The training set for generating initial values at the coarse time points has size $|\Xi_t|=36$, and the KL tolerance is $\varepsilon^{KL}=10^{-10}$.

\begin{figure}[htbp]
	\centering

    \includegraphics[height=6cm,width=6cm]{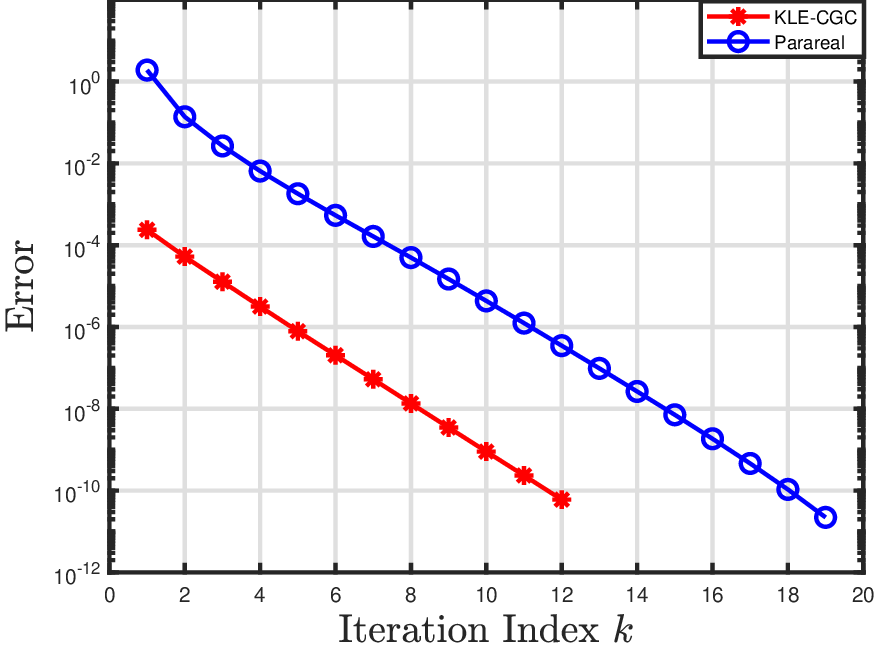}
	\qquad\qquad
    \includegraphics[height=6cm,width=6cm]{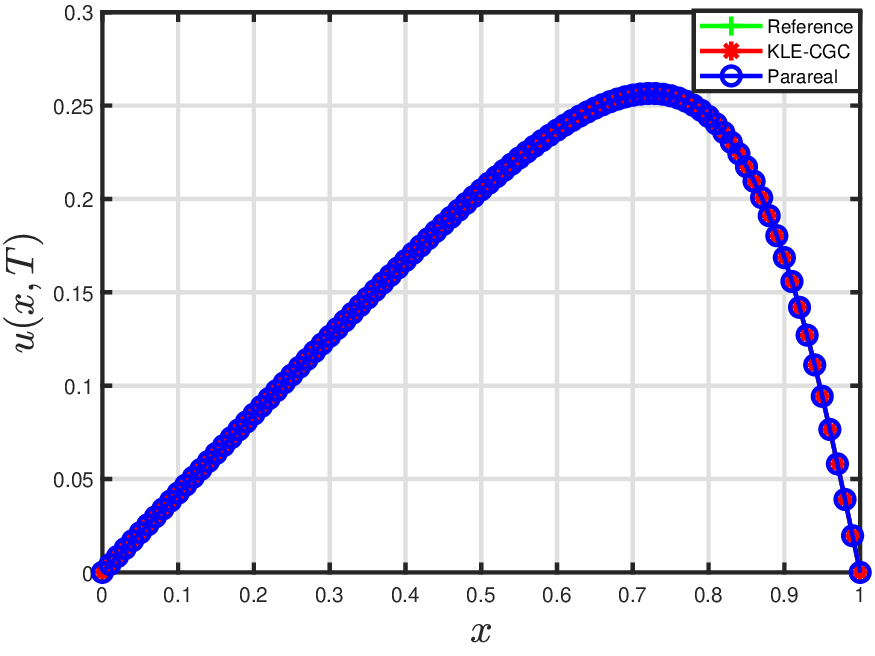}

	\caption{Comparison of the mean error between the parareal algorithm and the KLE-CGC algorithm (left) and the mean reference solution $u_N$, KLE-CGC solution $\mathrm{U}_N^k$, and parareal solution $u_N^k$ at the final time $t=T$ (right).}\label{b}
\end{figure}

In this experiment, we choose the KL expansion based CGC (KLE-CGC) as the main technique to compare with the standard parareal. To demonstrate the applicability of the KLE-CGC method to parameterized nonlinear problems, 1000 random parameter samples were tested. Fig. \ref{b} (left) shows the evolution of the mean error at the coarse time points for both algorithms versus the iteration count $k$. The error of KLE-CGC is consistently lower than that of Parareal, indicating that KLE-CGC provides a more accurate approximation to the reference solution at every iteration. Particularly in the initial iterations ($k=1,2$), the error of KLE-CGC is significantly lower, demonstrating its superior initial approximation capability, which benefits from its KLE-based reduced-order representation and covariance-guided coarse-grid correction. As iterations proceed, the errors of both methods decrease, but KLE-CGC maintains a lower error level, indicating better numerical stability even with coupled nonlinear and diffusive terms. Similarly, the error shows identical convergence rates for KLE-CGC and parareal; this agreement is consistent with the theoretical results.

Fig. \ref{b} (rigt) compares the solutions at the final time $t=T$. The reference solution shows the true state of the system at $T=2$. The KLE-CGC solution agrees well with both the reference and Parareal solutions, exhibiting excellent accuracy in capturing peak values and waveform details. This indicates that the KLE-CGC method also performs well in regions of strong nonlinearity.

\subsection{Allen-Cahn Equation}
We now consider the one-dimensional Allen-Cahn equation with a nonlinear term $f(u) = u^3 - u$:
\begin{equation*}
\begin{cases}
\frac{\partial u(x,t)}{\partial t} - \epsilon u_{xx} + f(u) = 0, & (x,t) \in (-1,1) \times (0,T), \\
u(x,0) = 0.53x + 0.47 \sin(-1.5\pi x), & x \in (-1,1), \\
u(-1,t) = -1, \, u(1,t) = 1, & t \in (0,T).
\end{cases}
\end{equation*}
Here, $\epsilon$ follows a truncated Gaussian distribution with original mean 0.53, original standard deviation 0.15, and truncation interval $[0.06,1]$, denoted $\epsilon \sim \mathcal{TN}(0.53, 0.15^2; 0.06, 1)$. The final time is $T=30$.

For spatial discretization, central differencing is used with grid size $\Delta x = 1/128$. For time discretization, the backward Euler method is applied with a coarse time step fixed at $\Delta T = 1$; the $\mathcal{F}$ propagator also uses backward Euler. The reference solution $\{ u_n\}^{N}_{n=1}$ (at the coarse time points) is computed sequentially using a uniform fine time step $\Delta t = \Delta T / 48$. The training set for generating initial values at the coarse time points has size $|\Xi_t|=10$, and the KL tolerance is $\varepsilon^{KL}=10^{-10}$.

An important feature of this equation is its Lyapunov energy functional,
\begin{equation*}
E(t) = \int_{-1}^{1} \left( \frac{\epsilon}{2} |\nabla u|^2 + F(u) \right) dx, \quad \text{where} \quad F(u) := \frac{(u^2 - 1)^2}{4},
\end{equation*}
which satisfies the decay property \( E(t) \leq E(s) \) for any \( t \geq s > 0 \).

\begin{figure}[htbp]
	\centering
	\includegraphics[height=6cm,width=6cm]{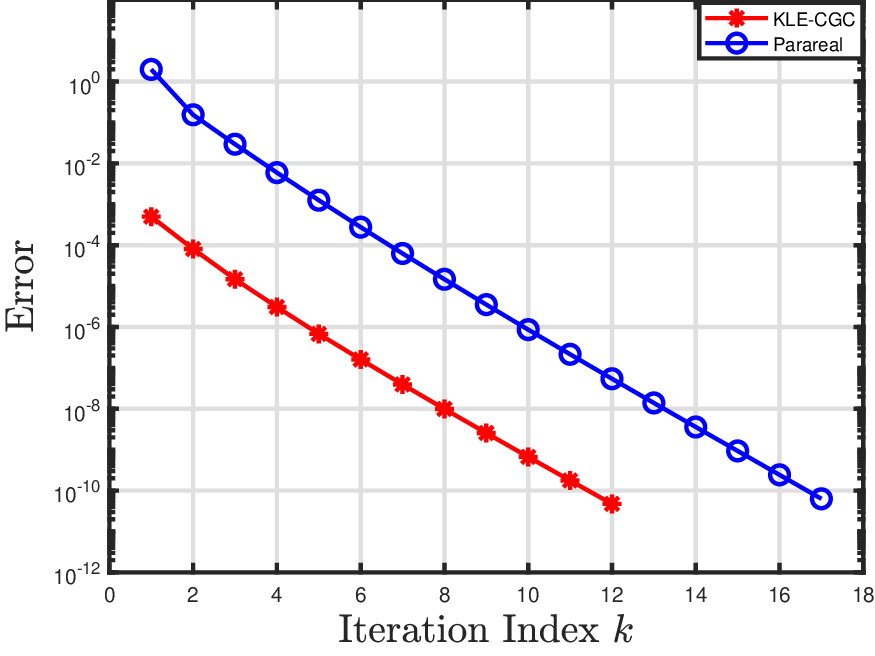}
	\qquad\qquad
	\includegraphics[height=6cm,width=6cm]{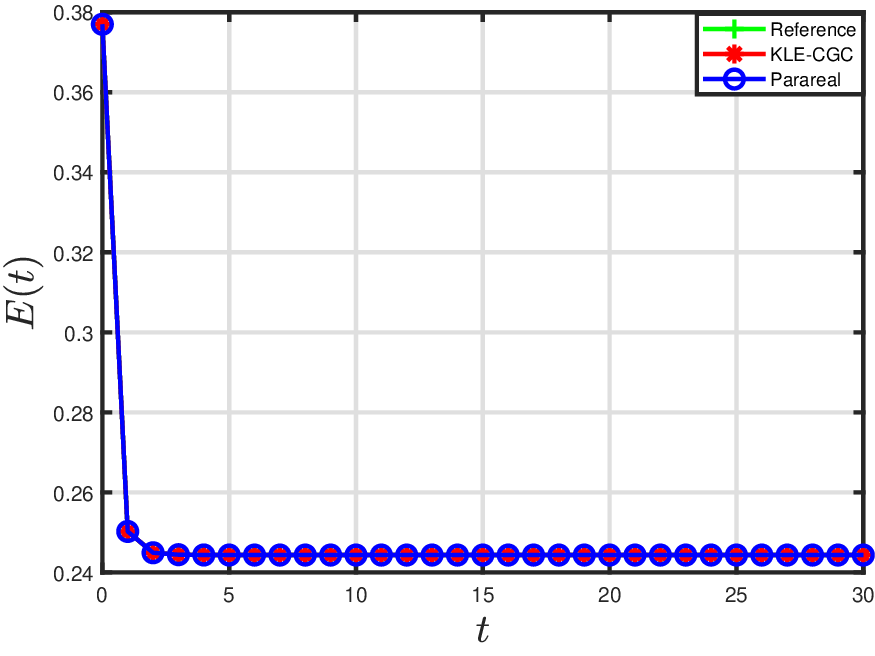}
	\caption{Comparison of the mean error (left) and the evolution of the mean energy $E(t)$ (right) between the parareal algorithm and the KLE-CGC algorithm.}\label{ac}
\end{figure}

In this experiment, KLE-CGC is also chosen as the main technology. Similarly, to demonstrate the high efficiency and accuracy of the KLE-CGC method, 1000 random parameter samples were tested. Fig. \ref{ac} (left) shows that KLE-CGC significantly outperforms Parareal at all iteration steps, maintaining consistently lower errors. Notably, at the first iteration ($k=1$), KLE-CGC already achieves high accuracy, while Parareal requires multiple iterations to improve gradually. This suggests that the KLE-based reduced-order representation more effectively captures the main dynamics of the system. As iterations proceed, both errors decrease, but KLE-CGC maintains its advantage, indicating greater robustness for strongly nonlinear systems.

 The energy $E(t)$ (right) decreases monotonically over time, consistent with theoretical expectations. The energy evolution of the KLE-CGC solution highly agrees with that of the reference and Parareal solutions, indicating that it not only accurately captures the functional form of the solution but also preserves the physical property (energy decay) of the system. This demonstrates that the method is particularly suitable for systems with an energy structure, maintains physical constraints, and performs robustly in long-time integration.

\section{Conclusion}\label{con}
A principal contribution of this work is the development of a novel hybrid parallel algorithm, termed KLE-PCGC, designed to significantly enhance the efficiency of traditional parallel-in-time methods when applied to problems with stochastic initial conditions. This methodology represents a strategic integration of parallel time integration and uncertainty quantification techniques. By employing the Karhunen-Lo\`{e}ve expansion for the low-dimensional parameterization of random fields and constructing a generalized Polynomial Chaos spectral surrogate model, the approach enables rapid and accurate prediction of the solution behavior. The central innovation lies in leveraging this prediction to provide high-quality initial guesses for the coarse grid, thereby directly addressing the fundamental bottleneck of slow convergence in the standard parareal algorithm caused by poor initial randomization.

Beyond the algorithmic construction, a rigorous convergence analysis is established, mathematically proving that the proposed KLE-PCGC framework retains the theoretical convergence rate of the standard parareal algorithm, which solidifies its theoretical foundation. Extensive numerical experiments comprehensively validate the superior performance of the method: compared to conventional approaches, KLE-PCGC achieves rapid convergence with considerably fewer iterations, ensuring high numerical accuracy while substantially improving parallel computational efficiency. Furthermore, the algorithm demonstrates remarkable generality, having been successfully applied to nonlinear systems and evolution equations possessing specific energy structures, thus providing a powerful and reliable tool for the high-precision and efficient parallel simulation of a broad class of complex dynamical systems.

In summary, this work successfully constructs a novel numerical framework that is theoretically sound, computationally efficient, and widely applicable by deeply embedding advanced dimension-reduction and surrogate modeling techniques from uncertainty quantification into a parallel-in-time architecture. It offers a systematic solution for the rapid simulation of parametrized dynamical systems.


\begin{thebibliography}{1}

\bibitem{J.L.L} J. L. Lions, Y. Maday and G. Turinici, A Parareal in time discretization of PDEs, C.R. Acad. Sci. Paris,  332: 661--668, 2001.

\bibitem{M.J} M. J. Gander and S. Vandewalle, Analysis of the parareal time-parallel time-integration
method, SIAM J. Sci. Comput., 29: 556--578, 2007.

\bibitem{S.L1} S. L. Wu, Convergence analysis of some second-order parareal algorithms, IMA J. Numer. Anal., 35: 1315--1341, 2015.

\bibitem{T.R.M} T.R. Mathew, M. Sarkis, and C.E. Schaerer, Analysis of block parareal preconditioners for parabolic optimal controal problems, SIAM J. Sci. Comput., 32: 1180--1200, 2010.

\bibitem{S.L2} S.L. Wu, Convergence analysis of some second-order parareal algorithms, IMA J. Numer. Anal., 35: 1315--1341, 2015.

\bibitem{S.L3} S.L. Wu and T. Zhou, Convergence analysis for three parareal solvers, SIAM J. Sci. Comput., 37: A970--A992, 2015.

\bibitem{S.L} S. L. Wu, Toward parallel coarse grid correction for the parareal algorithm, SIAM J. Sci. Comput., 40: A1446--A1472, 2018.

\bibitem{Ref2} S. L. Wu, An efficient parareal algorithm for a class of time-dependent problems with fractional Laplacian, Appl. Math. Comput., 307: 329--341, 2017.

\bibitem{Penk} K. Pentland,  M. Tamborrino, and D. A. L. C Samaddar,  Stochastic parareal: an application of probabilistic methods to time-parallelization, SIAM J. Sci. Comput., 45(3): S82-S102, 2023.

\bibitem{M.L.M} M. L. Minion, R. Speck, M. Bolten, M. Emmett, and D. Ruprecht, Interweaving PFASST
and parallel multigrid, SIAM J. Sci. Comput., 37: S244--S263, 2015.

\bibitem{R.D.F} R. D. Falgout, S. Friedhoff, Tz. V. Kolev, S. P. MacLachlan, and J. B. Schroder,
Parallel time integration with multigrid, SIAM J. Sci. Comput., 36: C635--C661, 2014.

\bibitem{V.A.D} V. A. Dobrev, Tz. Kolev, N. A. Petersson, and J. B. Schroder, Two-level convergence theory for multigrid reduction in time (MGRIT), SIAM J. Sci. Comput., 39:  S501--S527, 2017.

\bibitem{Y.M} Y. Maday and O. Mula, An adaptive parareal algorithm, J. Comput. Appl. Math., 377:
112915, 2020.

\bibitem{C.S} C. Schwab and R.A. Todor, Karhunen-Lo\`{e}ve approximation of random fields by gen eralized fast multipole methods, J. Comput. Phys., 217: 100–122, 2006.

\bibitem{D.X1} D. Xiu and G. Karniadakis, The Wiener-Askey polynomial chaos for stochastic differential equations, SIAM J. Sci. Comput., 24 (2002), pp. 619–644

\bibitem{J.S} J. Son and Y. Du, Comparison of intrusive and nonintrusive polynomial chaos expansion-based approaches for high dimensional parametric uncertainty quantification and propagation, Comput. Chem. Eng., 134: 106685 2020.

\bibitem{R.G} R. G. Ghanem and P. D. Spanos, Stochastic Finite Elements: A Spectral Approach, SpringerVerlag, 1991.

\bibitem{D.X2} D. Xiu and G.E. Karniadakis, Modeling uncertainty in steady-state diffusion problems via generalized polynomial chaos, Comput. Methods Appl. Mech. Eng., 191(43):4927-4948, 2002.


\bibitem{J.M} J. M. Reynolds-Barredo, D. E. Newman, R. Sanchez, D. Samaddar, L. A. Berry, and W. R. Elwasif, Mechanisms for the convergence of time-parallelized, parareal turbulent
plasma simulations, J. Comput. Phys., 231: 7851--7867, 2012.


\bibitem{J.M1} J. M. Reynolds-Barredo, D. E. Newman, and R. Sanchez, An analytic model for the convergence of turbulent simulations time-parallelized via the parareal algorithm, J. Comput.
Phys., 255: 293--315, 2013.

\bibitem{Y.E} Y. Maday and E. M. Ronquist, Parallelization in time through tensorproduct space-time
solvers, C. R. Math. Acad. Sci. Paris, 346: 113--118, 2008.




\bibitem{F.K} F. Kwok and B. Ong, Schwarz waveform relaxation with adaptive pipelining, SIAM J. Sci.
Comput., 41: A339--A364, 2019.

\bibitem{Smolyak_63}
S.~Smolyak, Quadrature and interpolation formulas for tensor products of
  certain classes of functions, Dokl. Akad. Nauk SSSR, 148:1042--1045, 1963.



\bibitem{M.J2} 
M. J. Gander and L. Halpern, Time parallelization for nonlinear problems based on diagonalization, in Domain Decomposition Methods in Science and Engineering XXIII, Lect.
Notes Comput. Sci. Eng. 116, Springer, Cham, 163--170, 2017.



\bibitem{M.J1} M. J. Gander, J. L. Jiang, B. Song, and H. Zhang, Analysis of two parareal algorithms for
time-periodic problems, SIAM J. Sci. Comput., 35: A2393--A2415, 2013.
\end{thebibliography}
\end{document}